\documentclass[a4paper,10pt]{amsart}
\usepackage{amsmath,amssymb}
\usepackage{amscd}
\newtheorem{thm}{Theorem}[section]
\newtheorem{prop}[thm]{Proposition}
\newtheorem{lemma}[thm]{Lemma}
\newtheorem{cor}[thm]{Corollary}

\theoremstyle{remark}
\newtheorem{remark}[thm]{Remark}
\newcommand{\id}{{\rm{id}}}
\newcommand{\Ad}{{\rm{Ad}}}

\newcommand{\BC}{\mathbf C}

\newcommand{\BK}{\mathbf K}
\newcommand{\BB}{\mathbf B}
\newcommand{\la}{\langle}
\newcommand{\ra}{\rangle}

\newcommand{\Aut}{{\rm{Aut}}}
\newcommand{\Equi}{{\rm{Equi}}}
\newcommand{\GEqui}{{\rm{GEqui}}}
\newcommand{\Pic}{{\rm{Pic}}}
\newcommand{\GPic}{{\rm{GPic}}}
\newcommand{\Int}{{\rm{Int}}}
\newcommand{\Ker}{{\rm{Ker}}}
\newcommand{\Ima}{{\rm{Im}}}
\newtheorem{Def}{Definition}[section]

\title{The generalized Picard groups for finite dimensional $C^*$-Hopf algebra coactions on
unital $C^*$-algebras}
\author{Kazunori Kodaka}
\address{Department of Mathematical Sciences, Faculty of Science, Ryukyu
\endgraf
University, Nishihara-cho, Okinawa, 903-0213, Japan}
\address{\sl{E-mail address}: \rm{kodaka@math.u-ryukyu.ac.jp}}

\begin{document}
\maketitle
\begin{abstract}
We shall generalize the notion of the strong Morita equivalence for coactions of
a finite dimensional $C^*$-Hopf algebra on a unital $C^*$-algebra and define the Picard groups with
respect to the generalized strong Morita equivalence. We call them the generalized Picard groups for
coactions of a finite dimensional $C^*$-Hopf algebra on a unital $C^*$-algeba.
We shall investigate basic properties of the generalized Picard groups and clarify the relation between
the generalized Picard groups and the Picard groups for unital inclusions of unital $C^*$-algebras.
\end{abstract}
 
\section{Introduction}\label{sec:intro} In \cite {KT3:equivalence}, we introduced the notion of the strong Morita
equivalence for coactions of a finite dimensional $C^*$-Hopf algebra on a unital $C^*$-algebra and
in \cite {Kodaka:equivariance} we defined the finite dimensional $C^*$-Hopf algebra coaction equivariant
Picard group of a unital $C^*$-algebra. We note that this equivariant Picard group can be regarded as
the Picard group for the coaction of the finite dimensional $C^*$-Hopf algebra on the unital $C^*$-algebra.
Furthermore, in \cite {KT4:morita} and \cite {Kodaka:Picard} we introduced the notion of the strong Morita
equivalence for unital inclusions of unital $C^*$-algebras and defined the Picard groups for them. In this paper,
we shall generalize the notion of the strong Morita equivalence for coactions of a finite dimensional $C^*$-Hopf
algebra on a unital $C^*$-algebra and we shall define the Picard groups with respect to the generalized
strong Morita equivalence. We call them the {\it generalized} Picard groups for coactions of a finite dimensional
$C^*$-Hopf algebra on a unital $C^*$-algebra. Also, we shall give the similar results to \cite {KT3:equivalence},
\cite {Kodaka:equivariance}, \cite{KT4:morita} and \cite {Kodaka:Picard}. Furthermore, we shall clarify the
relation between the generalized Picard groups and the Picard groups for unital inclusions of unital
$C^*$-algebras.

\section{Preliminaries}\label{sec:pre} Let $H$ be a finite dimensional $C^*$-Hopf algebra. We denote
its comultiplication, counit and antipode by $\Delta$, $\epsilon$, and $S$, respectively.
We shall use Sweedler's notation $\Delta(h)=h_{(1)}\otimes h_{(2)}$ for any $h\in H$ which
surppresses a possible summation when we write comultiplications. We denote by $N$ the
dimension of $H$. Let $H^0$ be the dual $C^*$-Hopf algebra of $H$. We denote its comultiplication,
counit and antipode by $\Delta^0$, $\epsilon^0$ and $S^0$, respectively. There is the distinguished projection $e$
in $H$. We note that $e$ is the Haar trace on $H^0$. Also, there is the distinguished projection $\tau$ in $H^0$
which is the Haar trace on $H$. Since $H$ is finite dimensional, $H\cong \oplus_{k=1}^L M_{f_k}(\BC)$
and $H^0 \cong \oplus_{k=1}^K M_{d_k}(\BC)$ as $C^*$-algebras, where $M_n (\BC)$ is the $n\times n$-
matrix algebra over $\BC$. Let $\{v_{ij}^k  \, | \, k=1,2, \dots, L, \, i,j=1,2,\dots,f_k \}$ be a system of matrix units
of $H$. Let $\{w_{ij}^k \, | \, k=1,2\dots, K, \, i,j=1,2,\dots, d_k\}$ be a basis of $H$ satisfying Szyma\'nski and
Peligrad's \cite [Theorem 2.2,2]{SP:saturated}, which is called a system of
\sl
comatrix units
\rm
of $H$, that is, the dual basis of a system of matrix units of $H^0$. Also, let
$\{\phi_{ij}^k \, | \, k=1,2,\dots, K, \, i,j=1,2,\dots, d_k \}$ and
$\{\omega_{ij}^k \, | \, k=1,2,\dots, L , \, i,j=1,2,\dots, f_k \}$ be systems of matrix units and
comatrix units of $H^0$, respectively.
\par
Let $A$ be a $C^*$-algebra and $M(A)$ its multiplier algebra. Let $p, q$ be projections in $A$.
If $p$ and $q$ are Murray-von Neumann equivalent, then we denote it by $p\sim q$ in $A$.
We denote by $\id_A$ and $1_A$ the identity map on $A$ and the unit element in $A$,
respectively. We simply denote them by $\id$ and $1$ if no confusion arises. Let $\Aut (A)$
be the group of all automorphisms of $A$ and for any $\alpha\in \Aut (A)$, let $\underline{\alpha}$ be
the automorphism of $M(A)$ induced by $\alpha$. Also, for any twisted coaction $(\rho, u)$ of $H^0$ on $A$,
let $(\underline{\rho}, u)$ be the twisted coaction of $H^0$ on $M(A)$ induced by $(\rho, u)$.
\par
Let $A$ be a $C^*$-algebra and $(\rho, u)$ a twisted coaction of $H^0$ on $A$. Let $f^0$ be a
$C^*$-Hopf algebra automorphism of $H^0$. Let $(\rho_{f^0}, u_{f^0})$ be the twisted coaction of $H^0$
on $A$ induced by $(\rho, u)$ and $f^0$, which is defined by
$$
\rho_{f^0}=(\id\otimes f^0 )\circ\rho, \quad u_{f^0}=(\id\otimes f^0 \otimes f^0 )(u) .
$$
Let $(\sigma, v)$ be a twisted coaction of $H^0$ on a $C^*$-algebra $B$ and $(\sigma_{f^0}, v_{f^0})$ the 
twisted coaction of $H^0$ on $B$ induced by $(\sigma, v)$ and $f^0$. For any $C^*$-Hopf algebra automorphism
$f^0$ of $H^0$, let $f$ be the $C^*$-Hopf algebra automorphism of $H$ induced by $f^0$, which is defined by
$$
\phi(f^{-1}(h))=f^0 (\phi)(h)
$$
for any $h\in H$, $\phi\in H^0$.

\begin{lemma}\label{lem:equivalence}With the above notations, we suppose that $(\rho, u)$ and
$(\sigma, v)$ are strongly Morita equivalent. Then for any $C^*$-Hopf algebra automorphism $f^0$ of
$H^0$, $(\rho_{f^0}, u_{f^0})$ and $(\sigma_{f^0}, v_{f^0})$ are strongly Morita equivalent.
\end{lemma}
\begin{proof}Since $(\rho, u)$ and $(\sigma, v)$ are strongly Morita equivalent, 
there are an $A-B$-equivalence bimodule $X$ and a twisted coaction $\lambda$ of $H^0$
on $X$ with respect to $(A, B, \rho, u, \sigma, v)$ (See \cite {KT3:equivalence}).
Let $\lambda_{f^0}$ be the linear map from $X$ to $X\otimes H^0$ defined by $\lambda_{f^0}
=(\id\otimes f^0 )\circ\lambda$. By routine computations, we can see that $\lambda_{f^0}$
is a twisted coaction of $H^0$ on $X$ with respect to $(A, B, \rho_{f^0}, u_{f^0}, \sigma_{f^0}, v_{f^0})$.
Indeed, for any $a\in A$, $x\in X$,
$$
\lambda_{f^0}(a\cdot x)=(\id\otimes f^0 )(\lambda(a\cdot x))=(\id\otimes f^0 )(\rho(a)\cdot \lambda(x)) .
$$
Since $\rho(a)\in A\otimes H^0$ and $\lambda(x)\in X\otimes H^0$, we can write that
$$
\rho(a)=\sum_i a_i \otimes \phi_i  , \quad  \lambda(x)=\sum_j x_j \otimes\psi_j ,
$$
where $a_i \in A$, $x_j \in X$, $\phi_i ,\psi_j \in H^0$. Hence
\begin{align*}
\lambda_{f^0}(a\cdot x) & =(\id\otimes f^0 )(\sum_{i,j}(a_i \cdot x_j )\otimes\phi_i \psi_j )
=\sum_{i, j}(a_i \otimes x_j )\otimes f^0 (\phi_i )f^0 (\psi_j ) \\
& =(\id\otimes f^0 )(\rho(a))\cdot (\id\otimes f^0 )(\lambda(x))=\rho_{f^0 }(a)\cdot\lambda_{f^0 }(x) .
\end{align*}
Similarly $\lambda_{f^0}(x\cdot b)=\lambda(x)\cdot\sigma_{f^0}(b)$ for any $b\in B$, $x\in X$.
For any $x, y\in X$,
$$
\sigma_{f^0 }(\la x, y \ra_B )=(\id\otimes f^0 )(\sigma(\la x, y \ra_B )
=(\id\otimes f^0 )(\la \lambda(x), \lambda(y) \ra_{B\otimes H^0 })
$$
Since $\lambda(x), \lambda(y)\in X\otimes H^0 $, we can write that
$$
\lambda(x)=\sum_i x_i \otimes\phi_i , \quad \lambda(y)=\sum_j y_j \otimes\psi_j ,
$$
where $x_i , y_j \in X$, $\phi_i , \psi_j \in H^0$. Hence
\begin{align*}
\sigma_{f^0}(\la x, y \ra_B ) & =(\id\otimes f^0 )(\sum_{i, j}\la x_i \otimes\phi_i  , \, y_j \otimes \psi_j \ra_{B\otimes H^0} ) \\
& =\sum_{i, j}\la x_i , y_j \ra_B \otimes f^0 (\phi_i^* )f^0 (\psi_j ) \\
& =\la (\id\otimes f^0 )(\lambda(x)), \, (\id\otimes f^0 )(\lambda(y)) \ra_{B\otimes H^0} \\
& =\la \lambda_{f^0}(x), \, \lambda_{f^0}(y) \ra_{B\otimes H^0} .
\end{align*}
Similarly $\rho_{f^0}({}_A \la x, y \ra)={}_{A\otimes H^0} \la \lambda_{f^0}(x), \, \lambda_{f^0}(y) \ra$ for any $x, y\in X$.
Also, for any $x\in X$,
$$
(\id\otimes\epsilon^0 )(\lambda_{f^0}(x))=((\id\otimes\epsilon^0 )\circ(\id\otimes f^0 ))(\lambda(x)) .
$$
Since $\lambda(x)\in X\otimes H^0 $, we can write that $\lambda(x)=\sum_i x_i \otimes\phi_i$,
where $x_i \in X$, $\phi_i \in H^0$. Let $f$ be the $C^*$-Hopf algebra automorphism of $H$ induced by $f^0$.
Then
\begin{align*}
((\id\otimes\epsilon^0 )\circ \lambda_{f^0})(x) & =\sum_i (\id\otimes\epsilon^0 )(x_i \otimes f^0 (\phi_i ))
=\sum_i x_i f^0 (\phi_i )(1) \\
& =\sum_i x_i \phi_i (f^{-1}(1))=\sum_i x_i \phi_i (1) \\
& =(\id\otimes\epsilon^0 )(\sum_i x_i \otimes\phi_i )=(\id\otimes\epsilon^0 )(\lambda(x))=x .
\end{align*}
Hence $((\id\otimes\epsilon^0 )\circ\lambda_{f^0})(x)=x$ for any $x\in X$.
Furthermore, for any $x\in X$,
\begin{align*}
(\lambda_{f^0}\otimes\id)(\lambda_{f^0}(x)) & =((\id\otimes f^0 \otimes\id)\circ(\lambda\otimes\id)\circ(\id\otimes f^0 )
\circ\lambda)(x) \\
& =(\id\otimes f^0 \otimes f^0 )(((\lambda\otimes\id)\circ\lambda)(x)) \\
& =(\id\otimes f^0 \otimes f^0 )(u(\id\otimes\Delta^0 )(\lambda(x))v^* ) \\
& =u_{f^0}(\id\otimes\Delta^0 )((\id\otimes f^0 )(\lambda(x))v_{f^0}^* \\
& =u_{f^0 }(\id\otimes\Delta^0 )(\lambda_{f^0} (x))v_{f^0}^* .
\end{align*}
Therefore, $(\rho_{f^0}, u_{f^0})$ and $(\sigma_{f^0}, v_{f^0})$ are strongly Morita equivalent.
\end{proof}
Let $A$, $B$ be $C^*$-algebra. For any $A-B$-equivalence bimodule $X$, let $\widetilde{X}$ be the
dual $B-A$-equivalence bimodule. For any $x\in X$, we denote by $\widetilde{x}$ the element in $\widetilde{X}$
induced by $x\in X$.

\section{Generalized strong Morita equivalence for coactions}\label{sec:Morita}
\begin{Def}\label{Def:Morita}Let $(\rho, u)$ and $(\sigma, v)$ be twisted coactions of $H^0$
on $A$ and $B$, respectively. The twisted coaction $(\rho, u)$ is
\sl
generalized strongly Morita equivalent
\rm
to the twisted coaction $(\sigma, v)$ if there are $C^*$-Hopf algebra automorphisms $f^0 , g^0$
such that $(\rho_{f^0}, u_{f^0})$ is strongly Morita equivalent to $(\sigma_{g^0}, v_{g^0})$.
\end{Def}

\begin{lemma}\label{lem:relation} The generalized strong Morita equivalence for twisted coactions
of a finite dimensional $C^*$-Hopf algebra on $C^*$-algebras is an equivalence relation.
\end{lemma}
\begin{proof}
This is immediate by routine computations and Lemma \ref{lem:equivalence}.
Indeed, it suffices to show transitivity since the other conditions clearly hold.
Let $A, B, C$ be $C^*$-algebras and $(\rho, u), (\sigma , v), (\gamma, w)$ twisted coaction of
$H^0$ on $A, B, C$, respectively.
We suppose that $(\rho, u)$ is generalized strongly Morita equivalent to $(\sigma, v)$ and that
$(\sigma, v)$ is generalized strongly Morita equivalent to  $(\gamma, w)$. Then there are $C^*$-Hopf algebra
automorphisms $f_1^0 , f_2^0 $ and $g_1^0 , g_2^0$ of $H^0$ such that $(\rho_{f_1^0}, u_{f_1^0})$
is strongly Morita equivalent to $(\sigma_{f_2^0}, v_{f_2^0})$ and such that $(\sigma_{g_1^0} , v_{g_1^0})$
is strongly Morita equivalent to $(\gamma_{g_2^0}, w_{g_2^0})$. By Lemma \ref{lem:equivalence},
$(\sigma_{f_2^0}, v_{f_2^0})$ is strongly Morita equivalent to $(\gamma_{f_2^0 \circ (g_1^0 )^{ -1}\circ g_2^0 } \, , \,
w_{f_2^0 \circ (g_1^0 )^{-1}\circ g_2^0})$. Hence by \cite [Proposition 3.7]{KT3:equivalence}, $(\rho_{f_1^0}, u_{f_1^0})$
is strongly Morita equivalent to $(\gamma_{f_2^0 \circ (g_1^0 )^{ -1}\circ g_2^0 } \, , \,
w_{f_2^0 \circ (g_1^0 )^{-1}\circ g_2^0})$. Therefore $(\rho, u)$ is generalized strongly Morita equivalent to $(\gamma, w)$.
\end{proof}

Modifying \cite {Kodaka:equivariance}, we shall define the generalized Picard group for a twisted coaction of a finite
dimensional $C^*$-Hopf algebra on a $C^*$-algebra.
\par
Let $A$ be a $C^*$-algebra and $H$ a finite dimensional $C^*$-Hopf algebra with its dual $C^*$-Hopf algebra
$H^0$. Let $(\rho, u)$ be a twisted coaction of $H^0$ on $A$. Let $(X, \lambda, f^0 )$ be a triplet of
an $A-A$-equivalence bimodule $X$, a twisted coaction $\lambda$ of $H^0$ on $X$ with respect to
$(A, A, \rho, u, \rho_{f^0}, u_{f^0})$ and a $C^*$-Hopf algebra automorphism $f^0$ of $H^0$,
where $(\rho_{f^0}, u_{f^0})$ is  the twisted coaction of $H^0$ on $A$ induced by $(\rho, u)$ and $f^0$. Let
$\GEqui_H^{\rho, u}(A)$ be the set of all such triplets $(X, \lambda, f^0 )$ as above.
We define an equivalence relation $\sim$ in $\GEqui_H^{\rho, u}(A)$ as follows:
For $(X, \lambda, f^0 ), (Y, \mu, g^0 ) \in \GEqui_H^{\rho, u}(A)$, $(X, \lambda, f^0 )\sim (Y, \mu, g^0 )$
if and only if there is an $A-A$-equivalence bimodule isomorphism $\pi$ of $X$ onto $Y$
such that $\mu\circ\pi=(\pi\otimes\id)\circ\lambda$ and $f^0 =g^0$. We denote by $[X, \lambda, f^0 ]$
the equivalence class of $(X, \lambda, f^0 )$ in $\GEqui_H^{\rho, u}(A)$. Let
$\GPic_H^{\rho, u}(A)=\GEqui_H^{\rho, u}(A)/\! \sim$. We shall define a product in $\GPic_H^{\rho, u}(A)$ as follows:
Let $(X, \lambda, f^0 ), (Y, \mu, g^0 )\in \GEqui_H^{\rho, u}(A)$. Let $\mu_{f^0}$ be the twisted coaction of
$H^0$ on $Y$ induced by $\mu$ and $f^0$. Then by the proof of Lemma \ref {lem:equivalence}, $\mu_{f^0}$
is the twisted coaction of $H^0$ on $Y$ with respect to $(A, A, \rho_{f^0}, u_{f^0}, \rho_{f^0 \circ g^0 }, u_{f^0 \circ g^0})$
and $(\rho_{f^0}, u_{f^0})$ is strongly Morita equivalence to $(\rho_{f^0 \circ g^0 }, u_{f^0 \circ g^0 })$. Hence
$$
(X\otimes_A Y , \lambda\otimes\mu_{f^0}, f^0 \circ g^0 )\in \GEqui_H^{\rho, u}(A) ,
$$
where $\lambda\otimes\mu_{f^0}$ is the twisted coaction of $H^0$ on $X\otimes_A Y$ induced by the
action $`` \cdot_{\lambda\otimes\mu_{f^0}} "$ on $X\otimes_A Y$ defined in \cite [Proposition 3.7]{KT3:equivalence}.
We define the product of $[X, \lambda, f^0 ], [Y, \mu, g^0 ]\in\GPic_H^{\rho, u}(A)$ by
$$
[X, \lambda, f^0 ][Y, \mu, g^0 ]=[X\otimes_A Y , \lambda\otimes\mu_{f^0} , f^0 \circ g^0 ]\in\GPic_H^{\rho, u}(A) .
$$
We show that the above product in $\GPic_H^{\rho, u}(A)$ is well-defined. Let $(\rho, u), (\sigma, v), (\gamma, w)$ be
twisted coactions of $H^0$ on $C^*$-algebras $A, B, C$, respectively. We suppose that there are a twisted
coaction $\lambda$ of
$H^0$ on an $A-B$-equivalence bimodule $X$ with respect to $(A, B, \rho, u, \sigma, v)$ and a twisted coaction $\mu$
of $H^0$ on a $B-C$-equivalence bimodule $Y$with respect to $(B, C, \sigma, v, \gamma, w)$, respectively.

\begin{lemma}\label{lem:define}With the above notations, let $\lambda'$ be a twisted coaction of $H^0$ on
an $A-B$-equivalence bimodule $X'$ with respect to $(A, B, \rho, u, \sigma, v)$.
We suppose that there is an $A-B$-equivalence bimodule
isomorphism $\pi$ of $X$ onto $X'$ such that $\lambda' \circ\pi=(\pi\otimes\id_{H^0})\circ\lambda$. Then $\pi\otimes\id_Y$
is an $A-C$-equivalence bimodule isomorphism of $X\otimes_A Y$ onto $X' \otimes_A Y$ such that
$$
(\lambda' \otimes\mu )\circ(\pi\otimes\id_Y )=(\pi\otimes\id_Y \otimes\id_{H^0})\circ(\lambda\otimes\mu) .
$$
\end{lemma}
\begin{proof}It is clear that $\pi\otimes\id_Y$ is an $A-C$-equivalence bimodule isomorphism of $X\otimes_A Y$
onto $X' \otimes_A Y$. We have only to show that
$$
(\lambda' \otimes\mu)\circ(\pi\otimes\id_Y )=(\pi\otimes\id_Y \otimes\id_{H^0})\circ(\lambda\otimes \mu) .
$$
That is, we show that
$$
h\cdot_{\lambda' \otimes\mu}(\pi\otimes\id_Y )(x\otimes y)=(\pi\otimes\id_Y )(h\cdot_{\lambda\otimes\mu}(x\otimes y))
$$
for any $x\in X$, $y\in Y$ and $h\in H$. Indeed, since
\begin{align*}
h\cdot_{\lambda' \otimes\mu}(\pi\otimes\id)(x\otimes y) & =[h_{(1)}\cdot_{\lambda'}\pi(x)]\otimes[h_{(2)}\cdot_{\mu}y] \\
& =\pi([h_{(1)}\cdot_{\lambda} x])\otimes[h_{(2)}\cdot_{\mu}y] \\
& =(\pi\otimes\id_Y )([h_{(1)}\cdot_{\lambda}x]\otimes[h_{(2)}\cdot_{\mu}y])
\end{align*}
for any $x\in X$, $y\in Y$, $h\in H$. Therefore, we obtain the conclusion.
\end{proof}

Similarly, we obtain the following lemma:

\begin{lemma}\label{lem:similar}With the above notations, let $\mu'$ be a twisted coaction of $H^0$ on
a $B-C$-equivalence bimodule $Y'$ with respect to $(B, C, \sigma, v, \gamma, w)$.
We suppose that there is a $B-C$-equivalence bimodule isomorphism $\pi$ of $Y$
onto $Y'$ such that $\mu' \circ \pi=(\pi\otimes\id_{H^0})\circ\mu$. Then $\id_X \otimes\pi$ is
an $A-C$-equivalence bimodule isomorphism of $X\otimes_A Y$ onto $X\otimes_A Y'$
such that 
$$
(\lambda\otimes\mu' )\circ(\id_X \otimes\pi)=(\id_X \otimes\pi\otimes\id_{H^0})\circ(\lambda\otimes\mu' ) .
$$
\end{lemma}

Let $f^0$ be a $C^*$-Hopf algebra automorphism. Let $(\rho, u)$, $(\sigma, v)$ be as above.
Also, let $\lambda'$ and $X'$ be as Lemma \ref{lem:define}.

\begin{lemma}\label{lem:define2}With the above notations, we suppose that there is an $A-B$-equivalence
bimodule isomorphism $\pi$ of $X$ onto $X'$ such that 
$\lambda' \circ\pi=(\pi\otimes\id_{H^0})\circ\lambda$. Then $\lambda_{f^0}'\circ\pi=(\pi\otimes\id_{H^0})\circ\lambda_{f^0}$.
\end{lemma}
\begin{proof}
We can prove the lemma by routine computations. Indeed,
\begin{align*}
\lambda_{f^0}' \circ\pi & =(\id\otimes f^0 )\circ\lambda' \circ\pi=(\id\otimes f^0 )\circ(\pi\otimes\id_{H^0})\circ\lambda \\
& (\pi\otimes\id_{H^0})\circ(\id\otimes f^0 )\circ\lambda =(\pi\otimes\id_{H^0})\circ\lambda_{f^0} .
\end{align*}
\end{proof}

\begin{lemma}\label{lem:define3}The product in $\GPic_H^{\rho, u}(A)$ defined by
$$
[X, \lambda, f^0 ][Y, \mu, g^0 ]=[X\otimes_A Y, \lambda\otimes\mu_{f^0}, f^0 \circ g^0 ]
$$
for any $(X, \lambda, f^0 ), (Y, \mu, g^0 )\in \GEqui_H^{\rho, u}(A)$ is well-defined.
\end{lemma}
\begin{proof}This is immediate by Lemmas \ref {lem:define}, \ref {lem:similar} and \ref {lem:define2}.
\end{proof}

We regard $A$ as an $A-A$-equivalence bimodule in the usual way. We sometimes denote it by ${}_A A_A$.
We also regard a twisted coaction $(\rho, u)$ of $H^0$ on $A$ as a twisted coaction of $H^0$ on the $A-A$-
equivalence bimodule ${}_A A_A$ with respect to $(A, A, \rho, u, \rho, u)$. Then $[{}_A A_A , \rho, \id_{H^0}]$
is the unit element in $\GPic_H^{\rho, u}(A)$. Let $(X, \lambda, f^0 )\in\GEqui_H^{\rho, u}(A)$.
Let $\widetilde{\lambda_{(f^0 )^{-1}}}$ be the linear map from $X$ to $\widetilde{X\otimes H^0}$ defined by
$$
\widetilde{\lambda_{(f^0)^{-1}}}(\widetilde{x})=[\lambda_{(f^0 )^{-1}}(x)]^{\widetilde{}}
=((\id\otimes (f^0 )^{-1})\circ\lambda)(x)^{\widetilde{}}
$$
for any $x\in X$. We note that $\widetilde{X\otimes H^0}$ is identified with $\widetilde{X}\otimes H^0$ by the map
$$
\widetilde{X\otimes H^0}\to\widetilde{X}\otimes H^0 : \widetilde{x\otimes\phi}\mapsto\widetilde{x}\otimes\phi^*
$$
as mentioned in \cite [Remarks 4.1 and 4.2]{Kodaka:equivariance}. Hence by easy computations, we can see that
$$
\widetilde{\lambda_{(f^0 )^{-1}}}(\widetilde{x})=((\id\otimes(f^0 )^{-1})\circ\widetilde{\lambda})(\widetilde{x})
$$
for any $x\in X$. Thus by Lemma \ref{lem:equivalence}, $\widetilde{\lambda_{(f^0 )^{-1}}}$ is a coaction of $H^0$
on $\widetilde{X}$ with respect to $(A, A, \rho, u, \rho_{(f^0 )^{-1}}, u_{(f^0 )^{-1}})$ and
$(\widetilde{X}, \widetilde{\lambda_{(f^0 )^{-1}}}, (f^0 )^{-1})\in\GEqui_H^{\rho, u}(A)$. Furthermore, by the product in
$\GPic_H^{\rho, u}(A)$,
\begin{align*}
[X, \lambda, f^0 ][\widetilde{X}, \widetilde{\lambda_{(f^0 )^{-1}}}, (f^0 )^{-1}]
& =[X\otimes_A \widetilde{X}, \lambda\otimes\widetilde{\lambda}, \id_{H^0}]
=[A, \rho, \id_{H^0}] , \\
[\widetilde{X}, \widetilde{\lambda_{(f^0 )^{-1}}}, (f^0 )^{-1}][X, \lambda, f^0 ]
& =[\widetilde{X}\otimes_A X , \widetilde{\lambda_{(f^0 )^{-1}}}\otimes\lambda_{(f^0 )^{-1}}, \id_{H^0} ]
=[A, \rho, \id_{H^0}]
\end{align*}
in $\GPic_H^{\rho, u}(A)$. Indeed, we identify $X\otimes_A \widetilde{X}$ and
$\widetilde{X}\otimes_A X$ with $A$ by the maps
\begin{align*}
\Phi_1 :  &  X\otimes_A \widetilde{X}\to A : x\otimes\widetilde{y}\mapsto {}_A \la x, y \ra , \\
\Phi_2 :  &  \widetilde{X}\otimes_A  X\to A : \widetilde{x}\otimes y\mapsto  \la x, y \ra_A , 
\end{align*}
respectively. Hence for any $h\in H$, $x, y\in X$,
\begin{align*}
\Phi_1 (h\cdot_{\lambda\otimes\widetilde{\lambda}}x\otimes\widetilde{y}) & =\Phi_1 ([h_{(1)}\cdot_{\lambda}x]
\otimes[S(h_{(2)}^* )\cdot_{\lambda}y]^{\widetilde{}}) \\
& ={}_A \la [h_{(1)}\cdot_{\lambda}x] \, ,\, [S(h_{(2)}^* )\cdot_{\lambda}y] \ra  \\
& =h\cdot_{\rho}{}_A \la x, y \ra =h\cdot_{\rho}\Phi_1 (x\otimes\widetilde{y}) , \\
\Phi_2 (h\cdot_{\widetilde{\lambda}_{(f^0 )^{-1}}\otimes\lambda_{(f^0 )^{-1}}}\widetilde{x}\otimes y) & =\Phi_2
([S(h_{(1)}^* )\cdot_{\lambda_{(f^0 )^{-1}}}x]^{\widetilde{}} \, , \, [h_{(2)}\cdot_{\lambda_{(f^0 )^{-1}}}y] ) \\
& =h\cdot_{\rho} \la x, y \ra_A =h\cdot_{\rho}\Phi_2 (\widetilde{x}\otimes y)
\end{align*}
since $\lambda_{(f^0 )^{-1}}$ is a coaction of $H^0$ on $X$ with respect to $(A, A, \rho_{(f^0 )^{-1}}, u_{(f^0 )^{-1}}, \rho, u)$.
Therefore $[\widetilde{X},\, \widetilde{\lambda_{(f^0 )^{-1}}}, \,(f^0 )^{-1}]$ is the inverse element of $[X, \lambda, f^0 ]$
in $\GPic_H^{\rho, u}(A)$. By the above product, $\GPic_H^{\rho, u}(A)$ is a group.
We call it the
\sl
generalized Picard group
\rm
of a twisted coaction of $(\rho, u)$ of $H^0$ on $A$. 

\begin{remark}\label{rem:isom}If $(\rho, u)$ and $(\sigma, \mu)$ are generalized strongly Morita equivalent
twisted coaction of $H^0$ on $A$ and $B$, respectively, then $\GPic_H^{\rho, u}(A)\cong\GPic_H^{\sigma, v}(B)$.
Indeed, since $(\rho, u)$ and $(\sigma, v)$ are generalized strongly Morita equivalent, there are a $C^*$-Hopf
algebra automorphism $g^0$ of $H^0$, an $A-B$-equivalence bimodule $Y$ and a twisted coaction $\mu$
of $H^0$ on $Y$ such that $(A, B, \rho, u, \sigma_{g^0}, v_{g^0}, \mu , H^0 )$ is a twisted covariant system
by Lemmas \ref {lem:equivalence} and \ref {Def:Morita}. We define the map $\Phi$ from $\GPic_H^{\rho, u}(A)$ to
$\GPic_H^{\sigma, v}(B)$ by
$$
\Phi([X, \lambda, f^0 ])=[\widetilde{Y}\otimes_A X\otimes_A Y, \, \widetilde{\mu_{(g^0 )^{-1}}}\otimes\lambda_{(g^0 )^{-1}}
\otimes\mu_{(g^0 )^{-1}\circ f^0 }, \, (g^0 )^{-1}\circ f^0 \circ g^0 ]
$$
for any $(X, \lambda, f^0 )\in \GEqui_H^{\rho, u}(A)$. Then by routine computations, we can see that 
$\Phi$ is an isomorphism of  $\GPic_H^{\rho, u}(A)$ onto $\GPic_H^{\sigma, v}(B)$.
\end{remark}

Let $\Aut(H^0 )$ be the group of all $C^*$-Hopf algebra automorphisms of $H^0$. For any $f^0 \in\Aut (H^0 )$,
let $\Aut_H^{\rho, u, f^0}(A)$ be the set of all automorphisms $\alpha$ of $A$ satisfying that
$$
\rho\circ\alpha=(\alpha\otimes f^0 )\circ\rho =(\alpha\otimes\id)\circ\rho_{f^0} , \quad
(\underline{\alpha}\otimes f^0 \otimes f^0 )(u)=u .
$$
Let $\Aut_H^{\rho, u}(A)=\Aut_H^{\rho, u, \id_{H^0}}(A)$. We note that $\Aut_H^{\rho, u}(A)$ is a subgroup
of $\Aut(A)$.

\begin{remark}\label{rem:auto}(1) Let $f^0, g^0 \in \Aut(H^0 )$. For any $\alpha\in\Aut_H^{\rho, u, f^0 }(A)$,
$\beta\in\Aut_H^{\rho, u, g^0 }(A)$, $\alpha\circ\beta\in \Aut_H^{\rho, u , f^0 \circ g^0 }(A)$ by easy computations.
\newline
(2) By the definitions of $\Pic_H^{\rho, u}(A)$ and $\GPic_H^{\rho, u}(A)$, we have the inclusion map $\jmath$ :
$$
\jmath: \Pic_H^{\rho, u}(A)\to\GPic_H^{\rho, u}(A): [X, \lambda]\mapsto [X, \lambda, \id_{H^0}] .
$$
By easy computations, $\jmath$ is a monomorphism of $\Pic_H^{\rho, u}(A)$ into $\GPic_H^{\rho, u}(A)$.
We regard $\Pic_H^{\rho, u}(A)$ as a subgroup of $\GPic_H^{\rho, u}(A)$ by $\jmath$.
Then $\Pic_H^{\rho, u}(A)$ is a normal subgroup of $\GPic_H^{\rho, u}(A)$ by $\jmath$.
Indeed, for any $[X, \lambda]\in\Pic_H^{\rho, u}(A)$ and $[Y, \mu, f^0 ]\in\GPic_H^{\rho, u}(A)$,
\begin{align*}
[Y, \mu, f^0 ][X, \lambda, \id_{H^0}][Y, \mu, f^0 ]^{-1} & =[Y\otimes_A X , \mu\otimes\lambda_{f^0} , f^0 ]
[\widetilde{Y}, \widetilde{\mu_{(f^0 )^{-1}}}, f^0 ] \\
& =[Y\otimes_A X\otimes_A \widetilde{Y}\, , \, \mu\otimes\lambda_{f^0}
\otimes(\widetilde{\mu_{(f^0)^{-1}}})_{f^0} \, ,\, \id_{H^0}]
\end{align*}
in $\GPic_H^{\rho, u}(A)$. By the definition of $(\widetilde{\mu_{(f^0)^{-1}}})_{f^0}$, for any $y\in Y$,
$$
(\widetilde{\mu_{(f^0)^{-1}}})_{f^0}(\widetilde{y})=((\id\otimes f^0 )\circ(\id\otimes(f^0 )^{-1})\circ\widetilde{\mu})(\widetilde{y})
=\widetilde{\mu}(\widetilde{y}) .
$$
Thus 
$$
[Y, \mu, f^0 ][X, \lambda, \id_{H^0}][Y, \mu, f^0 ]^{-1}
=[Y\otimes_A  X\otimes_A \widetilde{Y}, \, \mu\otimes\lambda_{f^0}\otimes\widetilde{\mu} \, , \id_{H^0}] .
$$
in $\GPic_H^{\rho, u}(A)$. Hence $\Pic_H^{\rho, u}(A)$ is a normal subgroup of $\GPic_H^{\rho, u}(A)$.
Also, let $\eta$ be the map from $\GPic_H^{\rho, u}(A)$ to $\Aut (H^0 )$ defined by
$$
\eta ([X, \lambda, f^0 ])=f^0
$$
for any $[X, \lambda, f^0 ]\in \GPic_H^{\rho, u}(A)$. Then clearly $\eta$ is an epimorphism of $\GPic_H^{\rho, u}(A)$ onto
$\Aut (H^0 )$. By the definition of $\eta$, $\Ima\,\jmath=\Ker \, \eta$. Thus, we have the exact sequence
$$
1\longrightarrow\Pic_H^{\rho, u}(A)\overset{\jmath}\longrightarrow\GPic_H^{\rho, u}(A)\overset{\eta}\longrightarrow
\Aut(H^0 )\longrightarrow 1 .
$$
Furthermore, let $\kappa$ be the map from $\Aut(H^0 )$ to $\GPic_H^{\rho, u}(A)$ defined by
$$
\kappa(f^0 )=[{}_A A_A , \rho, f^0 ]
$$
for any $f^0 \in \Aut (H^0 )$. Then $\kappa$ is a homomorphism of $\Aut (H^0 )$ to
$\GPic_H^{\rho, u}(A)$ with $\eta \circ\kappa=\id$ on $\Aut (H^0 )$. Indeed, 
for any $f^0 , g^0 \in\Aut (H^0 )$,
$$
\kappa(f^0 )\kappa(g^0 )=[{}_A  A_A , \rho, f^0 ][{}_A  A_A , \rho, g^0 ]
=[A\otimes_A A \, , \, \rho\otimes\rho_{f^0} \, , \, f^0 \circ g^0 ]
$$
in $\GPic_H^{\rho, u}(A)$. Let $\pi$ be the $A-A$-equivalence bimodule isomorphism of
$A\otimes_A A$ onto ${}_A  A_A$ defined by $\pi(A\otimes b)=ab$ for any $a, b\in A$.
Then for any $h\in H$, $f^0 , g^0 \in \Aut (H^0 )$,
\begin{align*}
\pi(h\cdot_{\rho\otimes\rho_{f^0}}a\otimes b) &= \pi(h\cdot_{\rho\otimes\rho_{f^0}}ab\otimes 1)
=\pi([h_{(1)}\cdot_{\rho}ab]\otimes[h_{(2)}\cdot_{\rho_{f^0}}1]) \\
& =[h_{(1)}\cdot_{\rho}ab]\epsilon(h_{(2)})=h\cdot_{\rho}ab=h\cdot_{\rho}\pi(a\otimes b) .
\end{align*}
Hence $[A\otimes_A  A \, , \, \rho\otimes\rho_{f^0} \, , \, f^0 \circ g^0 ]=[{}_A A_A \, , \, \rho \, , \, f^0 \circ g^0 ]$ in
$\GPic_H^{\rho, u}(A)$. Thus
$$
\kappa(f^0 )\kappa(g^0 )=\kappa(f^0 \circ g^0 ) .
$$
It follows that $\kappa$ is a homomorphism of $\Aut (H^0 )$ to $\GPic_H^{\rho, u}(A)$. It is clear that
$\eta\circ\kappa=\id$ on $\Aut (H^0 )$. Therefore, $\GPic_H^{\rho, u}(A)$ is isomorphic to
a semi-dierct product of $\Pic_H^{\rho, u}(A)$ by $\Aut (H^0 )$.
\newline
(3) Generally, $\Aut_H^{\rho, u, f^0}(A)\cap\Aut_H^{\rho, u, g^0}(A)\ne\emptyset$ even if $f^0, g^0 \in\Aut(H^0 )$
with $f^0 \ne g^0$. Indeed, let $\rho_{H^0}^A$ be the trivial coaction of $H^0$ on $A$. Then for any $\alpha\in\Aut (A)$,
$f^0 \in\Aut (H^0 )$,
$$
(\rho_{H^0}^A \circ\alpha)(a)=\alpha(a)\otimes 1^0 , \quad
((\alpha\otimes f^0 )\circ\rho_{H^0}^A )(a)=\alpha(a)\otimes 1^0
$$
for any $a\in A$. Hence $\rho_{H^0}^A \circ\alpha=(\alpha\otimes f^0 )\circ\rho_{H^0}^A$. Thus for any $f^0 \in\Aut (H^0 )$,
$\Aut _H^{\rho_{H^0}^A, \, f^0}(A)=\Aut (A)$.
\end{remark}

Modifying \cite {BGR:linking} and \cite {Kodaka:equivariance}, for each $\alpha\in\Aut_H^{\rho, u, f^0}(A)$, we construct
the element $(X_{\alpha}, \lambda_{\alpha}, f^0 )\in\GEqui_H^{\rho, u, f^0}(A)$ as follows:
Let $\alpha\in\Aut_H^{\rho, u, f^0}(A)$.
Let $X_{\alpha}$ be the vector space $A$ with the obvious left action of $A$ on $X_{\alpha}$ and
the obvious left $A$-valued inner product, but we define the right action of $A$ on $X_{\alpha}$
by $x\cdot a=x\alpha(a)$ for
any $x\in X_{\alpha}$, $a\in A$ and define the right $A$-valued inner product by $\la x, y \ra_A =\alpha(x^* y)$ for
any $x, y\in X_{\alpha}$. Then by \cite {BGR:linking}, $X_{\alpha}$ is an $A-A$-equivalence bimodule.
Also, $\rho$ can be regarded as a linear map from $X_{\alpha}$ to an $A\otimes H^0 -A\otimes H^0$-equivalence bimodule
$X_{\alpha}\otimes H^0$. We denote it by $\lambda_{\alpha}$. By easy computations, $\lambda_{\alpha}$ is a twisted
coaction $H^0$ on $X_{\alpha}$ with respect to $(A, A, \rho, u, \rho_{f^0}, u_{f^0})$. Thus we obtain the map $\Phi_{f^0}$:
$$
\Phi_{f^0} : \Aut_H^{\rho, u, f^0}(A)\to \GPic_H^{\rho, u}(A): \alpha\mapsto [X_{\alpha}, \lambda_{\alpha}, f^0 ] .
$$
Let $\mathcal{G}=\cup_{f^0 \in\Aut(H^0 )}\Aut_H^{\rho, u, f^0}(A)$ and 
let $\Phi$ be the map from $\mathcal{G}$ to $\GPic_H^{\rho, u}(A)$ defined by
$$
\Phi (\alpha)=\Phi_{f^0}(\alpha)=[X_{\alpha}, \lambda_{\alpha}, f^0 ]
$$
for any $\alpha\in\Aut_H^{\rho, u, f^0}(A)$, where $f^0 \in\Aut (H^0 )$.
We note that by easy computations, $\mathcal{G}$ is a subgroup of $\Aut(A)$.

\begin{lemma}\label{lem:homo}With the above notations, the map $\Phi:
\mathcal{G}\to\GPic_H^{\rho, u}(A)$ is a homomorphism.
\end{lemma}
\begin{proof}Let $\alpha\in\Aut_H^{\rho, u, f^0}(A)$ and $\beta\in\Aut_H^{\rho, u, g^0}(A)$. Then
$\alpha\circ\beta\in\Aut_H^{\rho, u, f^0 \circ g^0}(A)$. Also,
$$
\Phi(\alpha)\Phi(\beta) =[X_{\alpha}, \lambda_{\alpha}, f^0 ][X_{\beta}, \lambda_{\beta}, g^0 ] 
=[X_{\alpha}\otimes_A X_{\beta}\, , \, \lambda_{\alpha}\otimes(\lambda_{\beta})_{f^0}\, , \, f^0 \circ g^0 ] .
$$
Let $\pi$ be the map from $X_{\alpha}\otimes_A X_{\beta}$ to $X_{\alpha\circ\beta}$ defined by
$$
\pi(x\otimes y)=x\alpha(y)
$$
for any $x\in X_{\alpha}$, $y\in X_{\beta}$. By easy computations, $\pi$ is an $A-A$-equivalence bimodule isomorphism
of $X_{\alpha}\otimes_A X_{\beta}$ onto $X_{\alpha\circ\beta}$.
Also, for any $h\in H$, $x, y\in A$,
\begin{align*}
\pi(h\cdot_{\lambda_{\alpha}\otimes(\lambda_{\beta})_{f^0}}x\otimes y)
& =\pi([h_{(1)}\cdot_{\lambda_{\alpha}}x]\otimes[h_{(2)}\cdot_{(\lambda_{\beta})_{f^0}}y]) \\
& =\pi([h_{(1)}\cdot_{\rho}x]\otimes[h_{(2)}\cdot_{\rho_{f^0}}y]) \\
& =[h_{(1)}\cdot_{\rho}x]\alpha([h_{(2)}\cdot_{\rho_{f^0}}y]) .
\end{align*}
Since $\rho\circ\alpha=\alpha\circ\rho_{f^0}$, we can see that
$$
\pi(h\cdot_{\lambda_{\alpha}\otimes(\lambda_{\beta})_{f^0}}x\otimes y) =[h_{(1)}\cdot_{\rho} x][h_{(2)}\cdot_{\rho}\alpha(y)] 
=h\cdot_{\rho}x\alpha(y)=h\cdot_{\rho}\pi(x\otimes y) .
$$
Hence
$$
\Phi(\alpha)\Phi(\beta)=[X_{\alpha\circ\beta} \, , \, \lambda_{\alpha\circ\beta} \, , \, f^0 \circ g^0 ]=\Phi(\alpha\circ\beta) .
$$
Therefore, we obtain the conclusion.
\end{proof}

Let $\Int_H^{\rho, u}(A)$ be the group of all generalized inner automorphisms $\Ad(v)$ satisfying that $v$ is a unitary
element in $M(A)$ with $\underline{\rho}(v)=v\otimes 1^0$ and that
$u(v\otimes 1^0 \otimes 1^0 )=(v\otimes 1^0 \otimes 1^0 )u$.
By easy computations, $\Int_H^{\rho, u}(A)$ is a normal subgroup of $\mathcal{G}$.
Indeed, for any $f^0\in\Aut(H^0 )$, $\alpha\in\Aut_H^{\rho, u, f^0}(A)$ and $\Ad(v)\in\Int_H^{\rho, u}(A)$,
$$
\alpha\circ\Ad(v)\circ\alpha^{-1}=\Ad(\underline{\alpha}(v))\circ\alpha\circ\alpha^{-1}=\Ad(\underline{\alpha}(v)) .
$$
Also,
\begin{align*}
\underline{\rho}(\underline{\alpha}(v)) & =((\underline{\alpha}\otimes\id)\circ\underline{\rho})(v)
=(\underline{\alpha}\otimes\id)(v\otimes 1^0 )=\underline{\alpha}(v)\otimes 1^0 , \\
u(\underline{\alpha}(v)\otimes 1^0 \otimes 1^0 ) & =(\underline{\alpha}\otimes f^0 \otimes f^0 )(u)
(\underline{\alpha}\otimes f^0 \otimes f^0 )(v\otimes 1^0 \otimes 1^0 ) \\
& =(\underline{\alpha}\otimes f^0 \otimes f^0 )(u(v\otimes 1^0 \otimes 1^0 )) \\
& =(\underline{\alpha}\otimes f^0 \otimes f^0 )((v\otimes 1^0 \otimes 1^0 )u) \\
& =(\underline{\alpha}(v)\otimes 1^0 \otimes 1^0 )u .
\end{align*}
Hence $\Ad(\underline{\alpha}(v))\in\Int_H^{\rho, u}(A)$. Thus $\Int_H^{\rho, u}(A)$ is a normal subgroup
of $\mathcal{G}$. We denote by $\imath$ the inclusion map from $\Int_H^{\rho, u}(A)$
to $\mathcal{G}$.

\begin{lemma}\label{lem:exact1}With the above notations, we have the exact sequence
$$
1\longrightarrow\Int_H^{\rho, u}(A)\overset{\imath}{\longrightarrow}\mathcal{G}
\overset{\Phi}{\longrightarrow}\GPic_H^{\rho, u}(A) .
$$
\end{lemma}
\begin{proof}We have to show that $\Ima\, \imath=\Ker \, \Phi$. It is clear that $\Ima \, \imath\subset\Ker \, \Phi$
by \cite [Proposition 5.1]{Kodaka:equivariance} and Remark \ref {rem:auto}(2). We show that
$\Ker \, \Phi\subset\Ima \, \imath$. Let $f^0 \in\Aut(H^0 )$ and $\alpha\in\Aut_H^{\rho, u, f^0 }(A)$.
We suppose that $\Phi(\alpha)=[X_{\alpha} \,,\,\lambda_{\alpha}, \, f^0 ]=[{}_A A_A \, ,\, \rho\, , \,\id_{H^0}]$ in
$\GPic_H^{\rho, u}(A)$. Then $f^0 =\id_{H^0}$ and by Remark \ref{rem:auto}(2), $[X_{\alpha}, \lambda_{\alpha}]
=[{}_A A_A, \rho]$ in $\Pic_H^{\rho, u}(A)$. Hence by \cite [Proposition 5.1]{Kodaka:equivariance},
there is a unitary element $v\in M(A)$ such that
$$
\alpha=\Ad(v), \quad \underline{\rho}(v)=v\otimes 1^0 , \quad u(v\otimes 1^0 \otimes 1^0 )=u(v\otimes 1^0 \otimes 1^0 ) .
$$
Hence $\Ker \, \Phi\subset \Ima \, \imath$. Therefore, we obtain the conclusion.
\end{proof}

Let $A$ be a unital $C^*$-algebra and $\BK$ the $C^*$-algebra of all compact operators on a countably infinite
dimensional Hilbert space. Let $A^s =A\otimes\BK$. Let $\rho$ be a coaction of $H^0$ on $A$ and $\rho^s$ the
coaction of $H^0$ on $A^s$ induced by $\rho$, that is, $\rho^s =\rho\otimes\id_{\BK}$. Let $\Phi$ be the map from
$\mathcal{G}$ to $\GPic_H^{\rho^s}(A^s )$ defined in the above.
\par
Modifying the proof of \cite [Lemma 5.4]{Kodaka:equivariance}, we shall show that $\Phi$ is surjective under some
assumptions.
\par
We suppose that $\widehat{\rho}(1\rtimes_{\rho}e)\sim(1\rtimes_{\rho}e)\otimes 1$ in $(A\rtimes_{\rho}H)\otimes H$.
Let $f^0 \in\Aut(H^0 )$. Then $\widehat{\rho_{f^0}}(1\rtimes_{\rho_{f^0}}e)\sim(1\rtimes_{\rho_{f^0}}e)\otimes 1$ in
$(A\rtimes_{\rho_{f^0}}H)\otimes H$, where $\widehat{\rho_{f^0}}$ is the dual coaction of $\rho_{f^0}$.
Indeed, by \cite [Lemma 6.1]{KT5:Hopf}, there is an isomorphism $\pi$ of
$A\rtimes_{\rho}H$ onto $A\rtimes_{\rho_{f^0}}H$ defined by
$$
\pi(a\rtimes_{\rho}h)=a\rtimes_{\rho_{f^0}}f(h)
$$
for any $a\in A$, $h\in H$, which satisfies that
$$
\widehat{\rho_{f^0}}\circ\pi=(\pi\otimes\id_{H^0})\circ(\id_{A\rtimes_{\rho}H}\otimes f)\circ\widehat{\rho} ,
$$
where $f$ is the $C^*$-Hopf algebra automorphism of $H$ induced by $f^0$, that is
$$
f^0 (\phi)(h)=\phi(f^{-1}(h))
$$
for any $h\in H$, $\phi\in H^0$. Hence since $\pi(1\rtimes_{\rho}e)=1\rtimes_{\rho_{f^0}}f(e)=1\rtimes_{\rho_{f^0}}e$,
\begin{align*}
\widehat{\rho_{f^0}}(1\rtimes_{\rho_{f^0}}e) & =((\pi\otimes\id)
\circ(\id_{A\rtimes_{\rho}H}\otimes f))(\widehat{\rho}(1\rtimes_{\rho}e)) \\
& \sim ((\pi\otimes\id)\circ(\id_{A\rtimes_{\rho}H}\otimes f))((1\rtimes_{\rho}e)\otimes 1) \\
& =(\pi\otimes\id)((1\rtimes_{\rho}e)\otimes 1)=(1\rtimes_{\rho_{f^0}}e)\otimes 1
\end{align*}
in $(A\rtimes_{\rho_{f^0}}H)\otimes H$. It follows that by \cite [Lemma 5.2]{Kodaka:equivariance},
\begin{align*}
\widehat{\rho^s}(1\rtimes_{\rho^s}e) & \sim (1\rtimes_{\rho^s}e)\otimes 1 \quad \text{in} \quad
(M(A^s )\rtimes_{\underline{\rho^s}}H)\otimes H , \\
\widehat{\rho_{f^0}^s}(1\rtimes_{\rho_{f^0}^s}e) & \sim (1\rtimes_{\rho_{f^0}^s}e)\otimes 1 \quad \text{in} \quad
(M(A^s )\rtimes_{\underline{\rho_{f^0}^s}}H)\otimes H .
\end{align*}
Let $[X, \lambda, f^0 ]$ be any element in $\GPic_H^{\rho^s , f^0}(A^s )$. Let
$$
X^{\lambda}=\{x\in X \, | \, \lambda(x)=x\otimes 1^0 \} .
$$
Since $X$ an $A^s -A^s$-eqivalence bimodule, by \cite [Lemma 3.8 and Theorem 4.9]{Kodaka:equivariance},
$X_{\lambda}$ is an $(A^s )^{\rho_{f^0}^s}-(A^s )^{\rho_{f^0}^s}$-equivalence bimodule.
On the other hand, $(A^s )^{\rho_{f^0}^s}=(A^s )^{\rho^s}$. Indeed, for any $a\in (A^s )^{\rho_{f^0}^s}$,
$(\id\otimes f^0 )(\rho^s (a))=a\otimes 1^0 $. Hence $\rho^s (a)=a\otimes 1^0$, that is, $a\in (A^s )^{\rho^s}$.
Also, for any $a\in (A^s )^{\rho^s}$,
$$
\rho_{f^0}^s (a) =(\id\otimes f^0 )(\rho^s (a))=(\id\otimes f^0 )(a\otimes 1^0 )=a\otimes 1^0 .
$$
Hence $a\in (A^s )^{\rho_{f^0}^s}$. It follows that $X^{\lambda}$ is an $(A^s )^{\rho^s}-(A^s )^{\rho^s}$-
equivalence bimodule. Let $C$ be the linking $C^*$-algebra for $X$ and $\gamma$ the coaction of $H^0$
on $C$ induced by $\rho^s$ and $\lambda$. Let $C^{\gamma}$ be the fixed point $C^*$-algebra of $C$
for $\gamma$. Then by \cite [Lemma 4.6]{Kodaka:equivariance}, $C^{\gamma}$ is isomorphic to $C_0$,
the linking $C^*$-algebra for $X^{\lambda}$. We identify $C^{\gamma}$ with $C_0$. Let
$$
p=\begin{bmatrix} 1_A \otimes 1_{M(\BK)} & 0 \\
0 & 0 \end{bmatrix} , \quad
q=\begin{bmatrix} 0 & 0 \\
0 & 1_A \otimes 1_{M(\BK)} \end{bmatrix} .
$$
Since $M(C)^{\underline{\gamma}}=M(C^{\gamma})$ by \cite [Lemmas 2.9 and 4.7]{Kodaka:equivariance},
$p$ and $q$ are full for $C^{\gamma}$ such that $w^* w=p$ , $ww^* =q$. We note that
$w\in M(C)$. Let $\alpha$ be the map on $A^s$
defined by
$$
\alpha(a)=w^* aw=w^* \begin{bmatrix} 0 & 0 \\
0 & a \end{bmatrix}w
$$
for any $a\in A^s$. By routine computations, $\alpha$ is an automorphism of $A^s$. Let $\pi$ be the linear
map from $X$ to $X_{\alpha}$ defined by
$$
\pi(x)=\begin{bmatrix} 0 & x \\
0 & 0 \end{bmatrix}w =p\begin{bmatrix} 0 & x \\
0 & 0 \end{bmatrix}wp
$$
for any $x\in X$. In the same way as in the proof of \cite [Lemma 3.3]{BGR:linking}, we can see that $\pi$ is an
$A^s -A^s$-equivalence bimodule isomorphism of $X$ onto $X_{\alpha}$. For any $a\in A^s$,
\begin{align*}
(\rho^s \circ\alpha)(a) & =\rho^s (w^* aw)=\gamma(w^* \begin{bmatrix} 0 & 0 \\
0 & a \end{bmatrix}w) \\
& =\underline{\gamma}(w^* )\begin{bmatrix} 0 & 0 \\
0 & \rho_{f^0}^s (a) \end{bmatrix}\underline{\gamma}(w) \\
& =(w^* \otimes 1^0 )\begin{bmatrix} 0 & 0 \\
0 &\rho_{f^0}^s (a) \end{bmatrix} (w\otimes 1^0 ) \\
& =(\alpha\otimes\id_{H^0 })(\rho_{f^0}^s (a))
\end{align*}
since $w\in M(C)^{\underline{\gamma}}$. Hence $\alpha\in\Aut_H^{\rho^s , \, f^0}(A)$. Furthermore,
\begin{align*}
(\lambda_{\alpha}\circ\pi)(x) & =\lambda_{\alpha}(\begin{bmatrix} 0 & x \\
0 & 0 \end{bmatrix}w)=\rho^s (\begin{bmatrix} 0 & x \\
0 & 0 \end{bmatrix}w) \\
& =\gamma(\begin{bmatrix} 0 & x \\
0 & 0 \end{bmatrix}w)=\begin{bmatrix} 0 & \lambda(x) \\
0 & 0 \end{bmatrix}(w\otimes 1^0 ) \\
& =(\pi\otimes\id)(\lambda(x)) ,
\end{align*}
where we identify $\BK\otimes H^0$ and $H^0 \otimes\BK$. Thus $\Phi(\alpha)=[X, \lambda, f^0 ]$.
Therefore, we obtain the following lemma:

\begin{lemma}\label{lem:surjection}With the above notations, we suppose that $\widehat{\rho}(1\rtimes_{\rho}e)
\sim(1\rtimes_{\rho}e)\otimes 1$ in $(A\rtimes_{\rho}H)\otimes H$. Then $\Phi$ is a surjective homomorphism of
$\mathcal{G}$ onto $\GPic_H^{\rho^s}(A^s )$.
\end{lemma}
\begin{proof}This is immediate by the above discussions and Lemma \ref{lem:homo}.
\end{proof}

\begin{prop}\label{prop:exact2}With the above notations, we have the exact sequence
$$
1\longrightarrow\Int_H^{\rho^s}(A^s )\overset{\imath}{\longrightarrow}\mathcal{G}
\overset{\Phi}{\longrightarrow}\GPic_H^{\rho^s}(A^s )\longrightarrow 1 .
$$
\end{prop}
\begin{proof}
This is immediate by Lemmas \ref{lem:exact1} and \ref{lem:surjection}.
\end{proof}

\section{Crossed products}\label{sec:cross}Let $H$ be a finite dimensional $C^*$-Hopf
algebra with its dual $C^*$-Hopf algebra $H^0$. Let $(\rho, u)$ be a twisted coaction of $H^0$ on a unital
$C^*$-algebra $A$. Let $(X, \lambda, f^0 )\in\GEqui_H^{\rho, u}(A)$. Then $\lambda$ is a twisted
coaction of $H^0$ on $X$ with respect to $(A, A, \rho, u, \rho_{f^0}, u_{f^0})$. Hence by
\cite [Section 4]{Kodaka:equivariance}, $\widehat{\lambda}$ is a coaction of $H$ on $X\rtimes_{\lambda}H$
with respect to $(A\rtimes_{\rho, u}H \, , \, A\rtimes_{\rho_{f^0}, u_{f^0}}H \, , \, \widehat{\rho} \, , \, \widehat{\rho_{f^0}})$,
where $\widehat{\rho}$ and $\widehat{\rho_{f^0}}$ are dual coactions of $(\rho, u)$ and $(\rho_{f^0}, u_{f^0})$,
which are coactions of $H$ on $A\rtimes_{\rho, u}H$ and $A\rtimes_{\rho_{f^0}, u_{f^0}}H$, respectively.
Also, by \cite [Lemma 6.1]{KT5:Hopf}, there is an isomorphism $\pi$ of $A\rtimes_{\rho, u}H$ onto
$A\rtimes_{\rho_{f^0}, u_{f^0}}H$ such that 
$$
\widehat{\rho_{f^0}}\circ\pi=(\pi\otimes\id)\circ(\id\otimes f)\circ\widehat{\rho}, \quad
\pi(a\rtimes_{\rho, u}h)=a\rtimes_{\rho_{f^0}, u_{f^0}}f(h) ,
$$
where $a\in A$, $h\in H$ and $f$ is a $C^*$-Hopf algebra automorphism of $H$ induced by $f^0$.
By the above isomorphism $\pi$, we can regard $\widehat{\lambda}$
as a coaction of $H$ on $X\rtimes_{\lambda}H$ with respect to
$(A\rtimes_{\rho, u}H \, , \, A\rtimes_{\rho, u}H \, , \, \widehat{\rho }\, , \, (\widehat{\rho})_f )$,
where $(\widehat{\rho})_f =(\id\otimes f)\circ\widehat{\rho}$. Thus we obtain the element
$$
(X\rtimes_{\lambda}H, \widehat{\lambda}, f)\in\GEqui_{H^0}^{\widehat{\rho}}(A\rtimes_{\rho, u}H) .
$$
Let $F$ be the map from $\GPic_H^{\rho, u}(A)$ to $\GPic_{H^0}^{\widehat{\rho}}(A\rtimes_{\rho, u}H)$
defined by
$$
F([X, \lambda, f^0 ])=[X\rtimes_{\widehat{\lambda}}H, \widehat{\lambda}, f ]
$$
for any $(X, \lambda, f^0 )\in\GEqui_H^{\rho, u}(A)$. We can see that $F$ is well-defined in a straightforward way.
In this section, we show that $F$ is an isomorphism of $\GPic_H^{\rho, u}(A)$ onto
$\GPic_{H^0}^{\widehat{\rho}}(A\rtimes_{\rho, u}H)$.
First, we show that $F$ is a homomorphism of $\GPic_H^{\rho, u}(A)$ to
$\GPic_{H^0}^{\widehat{\rho}}(A\rtimes_{\rho, u}H)$. Let $[X, \lambda, f^0 ], [Y, \mu, g^0 ]\in\GPic_H^{\rho, u}(A)$.
Then
\begin{align*}
F([X, \lambda, f^0 ][Y, \mu, g^0 ]) & =F([X\otimes_A Y \, ,\,  \lambda\otimes\mu_{f^0} \, , \, f^0 \circ g^0 ]) \\
& =[(X\otimes_A Y)\rtimes_{\lambda\otimes\mu_{f^0}}H \, , \, \widehat{\lambda\otimes\mu_{f^0}} \, , \, j ] ,
\end{align*}
where $j$ is the $C^*$-Hopf algebra automorphism of $H$ induced by $f^0 \circ g^0$. Then $j=f\circ g$. Indeed,
for any $h\in H$, $\phi\in H^0$,
$$
\phi(j^{-1}(h))=(f^0 \circ g^0 )(\phi(h))=g^0 (\phi)(f^{-1}(h))=\phi((g^{-1}\circ f^{-1})(h)) .
$$
Hence $j=f\circ g$. Thus
$$
F([X, \lambda, f^0 ][Y, \mu, g^0 ])
=[(X\otimes_A Y)\rtimes_{\lambda\otimes\mu_{f^0}}H \, , \, \widehat{\lambda\otimes\mu_{f^0}} \, , \, f\circ g] .
$$
By \cite [Lemmas, 7.1 and 7.2]{Kodaka:equivariance}, there is an
$A\rtimes_{\rho, u}H-A\rtimes_{\rho_{f^0 \circ g^0}, u_{f^0 \circ g^0}}H$-equivalence
bimodule isomorphism $\Phi$ of $(X\otimes _A Y)\rtimes_{\lambda\otimes\mu_{f^0}}H$
onto $(X\rtimes_{\lambda}H)\otimes_{A\rtimes_{\rho_{f^0}, u_{f^0}}H}(Y\rtimes_{\mu_{f^0}}H)$
such that
$$
\Phi(\phi\cdot_{\widehat{\lambda\otimes\mu_{f^0}}}(x\otimes y\rtimes_{\lambda\otimes\mu_{f^0}}h))
=\phi\cdot_{\widehat{\lambda}\otimes\widehat{\mu_{f^0}}}\Phi(x\otimes y\rtimes_{\lambda\otimes\mu_{f^0}}h)
$$
for any $x\in X$, $y\in Y$, $h\in H$ and $\phi\in H^0$. Then $\Phi$ is defined by
$$
\Phi(x\otimes y\rtimes_{\lambda\otimes\mu_{f^0}}h)=(x\rtimes_{\lambda}1)\otimes (y\rtimes_{\mu_{f^0}}h)
$$
for any $x\in X$, $y\in Y$ and $h\in H$. Let $\pi_A$ (resp. $\pi_A^{g^0}$) be the linear map from $A\rtimes_{\rho, u}H$
(resp. $A\rtimes_{\rho_{f^0 \circ g^0}, u_{f^0 \circ g^0}}H$ ) defined by
\begin{align*}
\pi_A (a\rtimes_{\rho, u}h) & =a\rtimes_{\rho_{f^0}, u_{f^0}}f(h) \\
(\text{resp.} \quad \pi_A^{g^0}(a\rtimes_{\rho_{g^0}, u_{g^0}}h) & =a\rtimes_{\rho_{f^0 \circ g^0}, u_{f^0 \circ g^0}}f(h) \,)
\end{align*}
for any $a\in A$, $h\in H$. Then by \cite [Lemma 6.1]{KT5:Hopf}, $\pi_A$ and $\pi_A^{g^0}$ are isomorphisms of
$A\rtimes_{\rho, u}H$ and $A\rtimes_{\rho_{g^0}, u_{g^0}}H$ onto $A\rtimes_{\rho_{f^0}, u_{f^0}}H$ and
$A\rtimes_{\rho_{f^0 \circ g^0}, u_{f^0 \circ g^0 }}H$, respectively. Also, let $\pi_Y$ be the linear map from
$Y\rtimes_{\mu}H$ to $Y\rtimes_{\mu_{f^0}}H$ defined by
$$
\pi_Y (y\rtimes_{\mu}h)=y\rtimes_{\mu_{f^0}}f(h)
$$
for any $y\in Y$, $h\in H$. Clearly $\pi_Y$ is surjective.

\begin{lemma}\label{lem:condition}With the above notations, the following conditions hold:
\newline
$(1)$ $\pi_Y ((a\rtimes_{\rho, u}h)\cdot (y\rtimes_{\mu}l))=\pi_A (a\rtimes_{\rho, u}h)\cdot \pi_Y (y\rtimes_{\mu}l)$,
\newline
$(2)$ $\pi_Y ((y\rtimes_{\mu}l)\cdot(a\rtimes_{\rho_{g^0}, u_{g^0}}h))
=\pi_Y (y\rtimes_{\mu}l )\cdot\pi_A^{g^0}(a\rtimes_{\rho_{g^0}, u_{g^0}}h)$,
\newline
$(3)$ ${}_{A\rtimes_{\rho_{f^0}, u_{f^0}}H} \la \pi_Y (y\rtimes_{\mu}h) \, , \, \pi_Y (z\rtimes_{\mu}l) \ra
=\pi_A ({}_{A\rtimes_{\rho, u}H} \la y\rtimes_{\mu}h \, , \, z\rtimes_{\mu}l \ra)$,
\newline
$(4)$ $\la \pi_Y (y\rtimes_{\mu}h) \, , \, \pi_Y (z\rtimes_{\mu}l \ra_{A\rtimes_{\rho_{f^0 \circ g^0}, u_{f^0 \circ g^0}}H}
= \pi_A^{g^0}( \la y\rtimes_{\mu}h \, , \, z\rtimes_{\mu}l \ra_{A\rtimes_{\rho_{f^0}, u_{f^0}}H})$,
\newline
$(5)$ $\widehat{\mu_{f^0}}\circ\pi_Y =(\pi_Y\otimes\id_H )\circ(\widehat{\mu})_f$
\newline
for any $a\in A$, $y, z \in Y$, $h, l\in H$.
\end{lemma}
\begin{proof}
We can prove Conditions (1)-(4) in a straightforward way by the definitions of $\pi_A$, $\pi_A^{g^0}$ and $\pi_Y$.
We prove Condition (5). For any $y\in Y$, $h\in H$,
$$
(\widehat{\mu_{f^0}}\circ\pi_Y )(y\rtimes_{\mu}h)=\widehat{\mu_{f^0}}(y\rtimes_{\mu_{f^0}}f(h))
=(y\rtimes_{\mu_{f^0}}f(h_{(1)}))\otimes f(h_{(2)}) .
$$
On the other hand,
\begin{align*}
((\pi_Y \otimes\id_H )\circ(\widehat{\mu})_f )(y\rtimes_{\mu}h) & =((\pi_Y \otimes\id_H )\circ (\id_{Y\rtimes_{\mu}H}\otimes f)
\circ \widehat{\mu})(y\rtimes_{\mu}h) \\
& =((\pi_Y \otimes\id_H )\circ(\id_{Y\rtimes_{\mu}H}\otimes f))((y\rtimes_{\mu}h_{(1)})\otimes h_{(2)}) \\
& =(\pi_Y \otimes \id_H )((y\rtimes_{\mu}h_{(1)})\otimes f(h_{(2)})) \\
& =(y\rtimes_{\mu_f^0}f(h_{(1)}))\otimes f(h_{(2)}) .
\end{align*}
Hence we obtain the conclusion.
\end{proof}

\begin{lemma}\label{lem:homo2}With the above notations, the map $F$ is a homomorphism of $\GPic_H^{\rho, u}(A)$
to $\GPic_{H^0}^{\widehat{\rho}}(A\rtimes_{\rho, u}H)$.
\end{lemma}
\begin{proof}
Let $[X, \lambda, f^0 ], [Y, \mu, g^0 ]\in\GPic_H^{\rho, u}(A)$. Then
$$
F([X, \lambda, f^0 ][Y, \mu, g^0 ])
=[(X\otimes_A Y)\rtimes_{\lambda\otimes\mu_{f^0}}H \, , \, \widehat{\lambda\otimes\mu_{f^0}} \, , \, f\circ g ] .
$$
Also,
\begin{align*}
F([X, \lambda, f^0 ])F([Y, \mu, g^0 ]) & =[X\rtimes_{\lambda}H \, , \, \widehat{\lambda} \, , \, f]
[Y\rtimes_{\mu}H \, , \, \widehat{\mu} \, , \, g ] \\
& =[(X\rtimes_{\lambda}H)\otimes_{A\rtimes_{\rho_{f^0}, u_{f^0}}H}(Y\rtimes_{\mu}H) \, , \, \widehat{\lambda}
\otimes\widehat{\mu_{f^0}} \, , \, f\circ g ] .
\end{align*}
By the discussions before Lemma \ref{lem:condition}, there is an
$A\rtimes_{\rho, u}H-A\rtimes_{\rho_{f^0 \circ g^0}, u_{f^0 \circ g^0}}H$-equivalence bimodule
isomorphism $\Phi$ of $(X\otimes _A Y)\rtimes_{\lambda\otimes\mu_{f^0}}H$ onto
$(X\rtimes_{\lambda}H)\otimes_{A\rtimes_{\rho_{f^0}, u_{f^0}}H}(Y\rtimes_{\mu_{f^0}}H)$
such that
$$
\Phi(\phi\cdot_{\widehat{\lambda\otimes\mu_{f^0}}}((x\otimes y)\rtimes_{\lambda\otimes\mu_{f^0}}h))
=\phi\cdot_{\widehat{\lambda}\otimes\widehat{\mu_{f^0}}}\Phi((x\otimes y)\rtimes_{\lambda\otimes\mu_{f^0}}h)
$$
for any $x\in X$, $y\in Y$, $h\in H$ and $\phi\in H^0$. Since we identify $A\rtimes_{\rho_{f^0}, u_{f^0}}H$
with $A\rtimes_{\rho, u}H$ in the tensor product $(X\rtimes_{\lambda}H)
\otimes_{A\rtimes_{\rho_{f^0}, u_{f^0}}H}(Y\rtimes_{\mu}H)$
by the isomorphism $\pi_A$, by Lemma \ref{lem:condition}, we can see that
$$
F([X, \lambda, f^0 ])F([Y, \mu, g^0 ])=F([X, \lambda, f^0 ][Y, \mu, g^0 ])
$$
Therefore, we obtain the conclusion.
\end{proof}
Next, we construct the inverse homomorphism of $F$ from $\GPic_{H^0}^{\widehat{\rho}}(A\rtimes_{\rho, u}H)$
to $\GPic_H^{\rho, u}(A)$ modifying \cite [Section 9]{Kodaka:equivariance}. By the above discussions, there is the
homomorphism of $\widehat{F}$ of $\GPic_H^{\widehat{\rho}}(A\rtimes_{\rho, u}H)$ to $\GPic_H^{\widehat{\widehat{\rho}}}
(A\rtimes_{\rho, u}H\rtimes_{\widehat{\rho}}H^0 )$ defined by
$$
\widehat{F}([Y, \mu, f])=[Y\rtimes_{\mu}H^0 , \widehat{\mu}, f^0 ]
$$
for any $[Y, \mu, f]\in\GPic_{H^0}^{\widehat{\rho}}(A\rtimes_{\rho, u}H)$. By \cite [Proposition 2.8]{Kodaka:equivariance},
there are an isomorphism $\Psi_A $ of $A\otimes M_N (\BC)$ onto $A\rtimes_{\rho, u}H\rtimes_{\widehat{\rho}}H^0$
and a unitary element $U\in (A\rtimes_{\rho, u}H\rtimes_{\widehat{\rho}}H^0 )\otimes H^0$ such that 
\begin{align*}
\Ad(U)\circ\widehat{\widehat{\rho}} & = (\Psi_A \otimes\id_{H^0})\circ(\rho\otimes\id_{M_N (\BC)})\circ\Psi^{-1} , \\
(\Psi_A \otimes\id_{H^0}\otimes\id_{H^0})(u\otimes I_N) & =(U\otimes 1^0 )(\widehat{\widehat{\rho}}\otimes\id_{H^0})(U)
(\id\otimes\Delta^0 )(U^* ) .
\end{align*}
Let $\overline{\rho}=(\Psi_A^{-1}\otimes\id_{H^0})\circ\widehat{\widehat{\rho}}\circ\Psi_A $. For any $[X, \lambda, f^0 ]
\in\GPic_H^{\widehat{\widehat{\rho}}}(A\rtimes_{\rho, u}H\rtimes_{\widehat{\rho}}H^0 )$, we construct an element
$[X_{\Psi_A} \, , \, \lambda_{\Psi_A} \, , \, f^0 ]\in\GPic_H^{\overline{\rho}}(A\otimes M_N (\BC))$ as follows:
Let $X_{\Psi_A}=X$ as vector spaces. For any $x, y\in X_{\Psi_A}$ and $a\in A\otimes M_N (\BC)$,
\begin{align*}
a\cdot_{\Psi_A} x =\Psi_A (a)\cdot x \quad & , \quad x\cdot_{\Psi_A} a =x\cdot\Psi_A (a) , \\
{}_{A\otimes M_N (\BC)} \la x, y \ra & =\Psi_A^{-1}({}_{A\rtimes_{\rho, u}H\rtimes_{\widehat{\rho}}H^0 } \la x, y \ra) , \\
\la x , y \ra_{A\otimes M_N (\BC)} & =\Psi_A^{-1}(\la x, y \ra_{A\rtimes_{\rho, u}H\rtimes_{\widehat{\rho}}}H^0 ) .
\end{align*}
We regard $\lambda$ as a linear map from $X_{\Psi_A}$ to $X_{\Psi_A}\otimes H^0 $. We denote it by $\lambda_{\Psi_A}$.
Then $(X_{\Psi_A} \, , \, \lambda_{\Psi_A} \, , \, f^0 )\in\GEqui_H^{\overline{\rho}}(A\otimes M_N (\BC))$.
By easy computations, the map
$$
\GPic_H^{\widehat{\widehat{\rho}}}(A\rtimes_{\rho, u}H\rtimes_{\widehat{\rho}}H^0 )\to
\GPic_H^{\overline{\rho}}(A\otimes M_N (\BC)): [X, \lambda, f^0 ]\mapsto [X_{\Psi_A} \, , \, \lambda_{\Psi_A} \, , \, f^0 ]
$$
is well-defined and it is an isomorphism of $\GPic_H^{\widehat{\widehat{\rho}}}(A\rtimes_{\rho, u}H\rtimes_{\widehat{\rho}})$
onto $\GPic_H^{\overline{\rho}}(A\otimes M_N (\BC))$. We denote by $G_1$ the above isomorphism. Furthermore,
the coaction $\overline{\rho}$ of $H^0$ on $A\otimes M_N (\BC)$ is exterior equivalent to the twisted coaction
$(\rho\otimes\id , u\otimes I_N )$ since
\begin{align*}
\rho\otimes\id_{M_N (\BC)} & =(\Psi_A^{-1}\otimes\id_{H^0})\circ\Ad(U)\circ\widehat{\widehat{\rho}}\circ\Psi_A
=\Ad(U_1 )\circ\overline{\rho} , \\
u\otimes I_N & =(U_1 \otimes 1^0 )(\overline{\rho}\otimes\id )(U_1 )(\id\otimes\Delta^0 )(U_1^*) ,
\end{align*}
where $U_1 =(\Psi_A^{-1}\otimes\id_{H^0})(U)$. Hence there is the isomorphism $G_2$ of $\GPic_H^{\overline{\rho}}
(A\otimes M_N (\BC))$ onto $\GPic_H^{\rho\otimes\id_{M_N (\BC)}, u\otimes I_N}(A\otimes M_n (\BC))$ defined by
$$
G_2 ([X, \lambda, f^0 ])=[X, \Ad(U_1 )\circ\lambda, f^0 ]
$$
for any $[X, \lambda, f^0 ]\in\GPic_H^{\overline{\rho}}(A\otimes M_N (\BC))$. Since $(\rho, u)$ is strongly Morita
equivalent to $(\rho\otimes\id_{M_N (\BC)}, u\otimes I_N )$, there is the isomorphism of $G_3$ of $\GPic_H^{\rho, u}(A)$
onto $\GPic_H^{\rho\otimes\id_{M_N (\BC)}, u\otimes I_N }(A\otimes M_N (\BC))$ defined by
$$
G_3 ([X, \lambda, f^0 ])=[X\otimes M_N (\BC)\, , \, \lambda\otimes\id_{M_N (\BC)} \, , \, f^0 ]
$$
for any $[X, \lambda, f^0 ]\in\GPic_H^{\rho, u}(A)$. Let $G=G_3 ^{-1}\circ G_2 \circ G_1 \circ \widehat{F}$.
Thus $G$ is a homomorphism of $\GPic_{H^0}^{\widehat{\rho}}(A\rtimes_{\rho, u}H)$ to $\GPic_H^{\rho, u}(A)$.

\begin{lemma}\label{lem:iso3}With the above notations, $G\circ F=\id$ on $\GPic_H^{\rho, u}(A)$.
\end{lemma}
\begin{proof}We prove the lemma modifying \cite [Proposition 7.4]{Kodaka:equivariance}. Let
$[X, \lambda, f^0 ]\in\GPic_H^{\rho, u}(A)$. By the definitions of $F, \widehat{F}, G_1, G_2 $,
\begin{align*}
(G_2 \circ G_1 \circ\widehat{F}\circ F)([X, \lambda, f^0 ]) & =(G_2 \circ G_1 \circ\widehat{F})
([X\rtimes_{\lambda}H \, , \, \widehat{\lambda} \, , \, f ]) \\
& =(G_2 \circ G_1 )([X\rtimes_{\lambda}H\rtimes_{\widehat{\lambda}}H^0 \, , \, \widehat{\widehat{\lambda}} \, , \, f^0 ]) \\
& =G_2 ([(X\rtimes_{\lambda}H\rtimes_{\widehat{\lambda}}H^0 )_{\Psi_A} \, , \, (\widehat{\widehat{\lambda}})_{\Psi_A} \, , \,
f^0 ]) \\
& =[(X\rtimes_{\lambda}H\rtimes_{\widehat{\lambda}}H^0 )_{\Psi_A} \, , \, \Ad(U_1 )
\circ(\widehat{\widehat{\lambda}})_{\Psi_A } \, , \, f^0 ] .
\end{align*}
Let $\Psi_X$ be the linear map from $X\otimes M_N (\BC)$ to $X\rtimes_{\lambda}H\rtimes_{\widehat{\lambda}}H^0 $
defined in \cite [Proposition 3.6]{Kodaka:equivariance} and we regard $\Psi_X$ as an $A\otimes M_N (\BC)
-A\otimes M_N (\BC)$-equivalence bimodule isomorphism of $X\otimes M_N (\BC)$ onto
$(X\rtimes_{\lambda}H\rtimes_{\widehat{\lambda}}H^0 )_{\Psi_A}$. Also, since
$$
\Ad(U)\circ\widehat{\widehat{\lambda}}=
(\Psi_X \otimes \id)\circ(\lambda\otimes\id)\circ\Psi_X^{-1}
$$
by \cite [Proposition 3.6]{Kodaka:equivariance}, for any $x\in A\otimes M_N (\BC)$,
\begin{align*}
(\Ad(U_1 )\circ(\widehat{\widehat{\lambda}})_{\Psi_A})(x) & =U_1 \cdot_{\Psi_A}(\widehat{\widehat{\lambda}})_{\Psi_A}
(x)\cdot_{\Psi_A}U_1^* =U\widehat{\widehat{\lambda}}(x)U^* \\
& =((\Psi_X \otimes\id)\circ(\lambda\otimes\id)\circ\Psi_X^{-1})(x) .
\end{align*}
Thus
$$
[(X\rtimes_{\lambda}H\rtimes_{\widehat{\lambda}}H^0 )_{\Psi_A} \, , \, 
\Ad(U_1 )\circ(\widehat{\widehat{\lambda}})_{\Psi_A} \, , \, f^0 ]=[X\otimes M_N (\BC) \, , \, \lambda\otimes\id \, , \, f^0 ]
$$
in $\GPic_H^{\rho\otimes\id_{M_N (\BC)}, u\otimes I_N}(A\otimes M_N (\BC))$. Since
$$
G_3 ([X, \lambda, f^0 ])=[X\otimes M_N (\BC) \, , \, \lambda\otimes\id_{M_N (\BC)} \, , \, f^0 ]
$$
in $\GPic_H^{\rho\otimes\id_{M_N (\BC)}, u\otimes I_N}(A\otimes M_N (\BC))$, we obtain the conclusion.
\end{proof}

\begin{thm}\label{thm:iso5}Let $H$ be a finite dimensional $C^*$-Hopf algebra
with its dual $C^*$-Hopf algebra $H^0$. Let $(\rho, u)$ be a twisted coaction of $H^0$
on a unital $C^*$-algebra $A$. Then $\GPic_H^{\rho, u}(A)
\cong\GPic_{H^0}^{\widehat{\rho}}(A\rtimes_{\rho, u}H)$, where $\widehat{\rho}$ is the dual coaction
of $(\rho, u)$.
\end{thm}
\begin{proof}This is immediate by Lemma \ref{lem:iso3} in the same way as in
the proof of \cite [Theorem 7.5]{Kodaka:equivariance}.
\end{proof}

\section{Preparation}\label{sec:prepararion}In this section, we prepare some results for the next section.
\par
Let $(\rho, u)$ and $(\sigma, v)$ be twisted coactions of $H^0$ on $A$ and let $C=A\rtimes_{\rho, u}H$ and
$D=A\rtimes_{\sigma, v}H$. We suppose that $A' \cap C=\BC1$. We also suppose that
the unital inclusions $A\subset C$ and
$A\subset D$ are strongly Morita equivalent with respect to a $C-D$-equivalence bimodule $Y$ and its
closed subspace ${}_A A_A$. Hence we regard $A$ as a closed subspace of $Y$.
Furthermore, since $A' \cap C=\BC1$,
by \cite [Lemma 3.1]{Kodaka:Picard}, there is the unique conditional expectation $F$ from $Y$ onto $A$ with respect to
$E_1^{\rho, u}$ and $E_1^{\sigma, v}$, where $E_1^{\rho, u}$ and $E_1^{\sigma, v}$ are
the canonical conditional expectations from $C$ and $D$ onto $A$, respectively defined by
$$
E_1^{\rho, u}(a\rtimes_{\rho, u}h)=\tau(h)a , \quad
E_1^{\sigma, v}(a\rtimes_{\sigma, v}h)=\tau(h)a
$$
for any $a\in A$, $h\in H$. By the proof of Rieffel \cite [Proposition 2.1]{Rieffel:rotation}, there is an isomorphism $\Psi$
of $D$ onto $C$ defined by
$$
\Psi(d)={}_C \la 1_A \cdot d \, , \, 1_A \ra
$$
for any $d\in D$, where $1_A$ is the unit element in $A$ and we regard $1_A$
as an element in $Y$. Let $C_{\Psi}$ be the $C-D$-equivalence bimodule induced by $C$ and $\Psi$,
that is, $C_{\Psi}=C$ as vector spaces and the left $C$-action and the left $C$-valued inner
product are defined in the usual way. We define the right $D$-action by $x\cdot d=x\Psi(d)$ for any
$x\in C$, $d\in D$ and define the right $D$-valued inner product by $\la x, y \ra_D =\Psi^{-1}(x^* y)$
for any $x, y\in C$.

\begin{lemma}\label{lem:s-iso}With the above notations, $Y\cong C_{\Psi}$ as $C-D$-equivalence bimodules.
\end{lemma}
\begin{proof}
We note that $y={}_C \la y, 1_A \ra \cdot 1_A$ for any $y\in Y$ since $\la 1_A , 1_A \ra_D =1_A$. Let $\eta$
be the map from $Y$ to $C_{\Psi}$ defined by
$$
\eta(y)={}_C \la y, 1_A \ra
$$
for any $y\in Y$. Then the map $c\in C_{\Psi} \mapsto c\cdot 1_A \in Y$ is the inverse map of
$\eta$. Hence $\eta$ is bijective. Clearly $\eta$ is linear. For any $y, z\in Y$,
\begin{align*}
\la \eta(y), \eta(z) \ra_D & =\Psi^{-1}({}_C \la y, 1_A \ra^* {}_C \la z, 1_A \ra )
=\Psi^{-1}({}_C \la 1_A , y \ra {}_C \la z, 1_A \ra) \\
& =\Psi^{-1}({}_C \la {}_C \la 1_A , y \ra\cdot z \, , \, 1_A \ra ) \\
& =\Psi^{-1}({}_C \la 1_A \cdot \la y, z \ra_D \, , \, 1_A \ra)=\la y, z \ra_D , \\
{}_C \la \eta(y), \eta(z) \ra & ={}_C \la {}_C \la y, 1_A \ra \, , \, {}_C \la z, 1_A \ra \ra
={}_C \la y, 1_A \ra \, {}_C \la 1_A , z \ra \\
& ={}_C \la {}_C \la y, 1_A \ra \cdot 1_A \, , \, z \ra
={}_C \la y\cdot \la 1_A , 1_A \ra_D \, , \, z \ra={}_C \la y, z \ra .
\end{align*}
Hence $\eta$ preserves the left $C$-valued and the right $D$-valued inner products. Therefore, we obtain the conclusion,
\end{proof}
Let $\widehat{\rho}$ and $\widehat{\sigma}$ be the dual coactions of $(\rho, u)$ and $(\sigma, v)$, respectively
and let $C_1 =C\rtimes_{\widehat{\rho}}H^0$, $D_1=D\rtimes_{\widehat{\sigma}}H^0$. Let $Y_1$ be the upward basic
construction of $Y$ for $F$. Let $\widehat{\Psi}$ be the isomorphism of $D_1$ onto $C_1$
defined by
$$
\widehat{\Psi}(T)=\Psi\circ T\circ\Psi^{-1}
$$
for any $T\in \BB_A (D)$, where we regard $C$ and $D$ as a $C_1 -A$-equivalence bimodule and $D_1 -A$-
equivalence bimodule using $E_1^{\rho, u}$ and $E_1^{\sigma, v}$, respectively and we identify $C_1$ and $D_1$
with the $C^*$-algebra $\BB_A (C)$, the $C^*$-algebra of all adjointable right $A$-module maps on $C$ and $\BB_A (D)$,
the $C^*$-algebra of all adjointable right $A$-module maps on $D$, respectively. Then by easy computations,
$$
\widehat{\Psi}|_D =\Psi, \quad \Psi\circ E_2^{\sigma, v}=E_2^{\rho, u}, \quad
\widehat{\Psi}(1\rtimes_{\widehat{\sigma}}\tau)
=1\rtimes_{\widehat{\rho}}\tau ,
$$
where $E_2^{\rho, u}$ and $E_2^{\sigma, v}$ are the conditional expectations from $C_1$
and $D_1$ onto $C$ and $D$, which are defined by
$$
E_2^{\rho, u}(c\rtimes_{\widehat{\rho}}\phi)=\phi(e)c, \quad E_2^{\sigma, v}(d\rtimes_{\widehat{\sigma}}\phi)=\phi(e)d
$$
for any $c\in C$, $d\in D$, $\phi\in H^0$, respectively. Let $\Phi_1$ be the map from $D_1$ to $C_1$ defined by
$$
\Psi_1 (d_1 )={}_{C_1} \la 1_A \cdot d_1 \, , \, 1_A \ra
$$
for any $d_1 \in D_1$. Then by routine computations, $\Psi_1 (d)=d$ for any $d\in D$ and
$E_2^{\rho, u}\circ \Psi_1 =\Psi\circ E_2^{\sigma, v}$. Indeed, we note that $1_A \in A$ is regarded as an element
$$
\sum_{i, j}u_i \otimes F(u_i^* \cdot 1 \cdot v_j )\otimes \widetilde{v_j}
$$
in $Y_1$, where $\{(u_i , u_i ^* )\}$ is a quasi-basis for $E_1^{\rho, u}$ and $\{(v_j , v_j )\}$ is a quasi-basis
for $E_1^{\sigma, v}$. Let $c, d\in D$. Then
\begin{align*}
& \Psi_1 ((c\rtimes_{\widehat{\sigma}}1^0 )(1\rtimes_{\sigma, v}1
\rtimes_{\widehat{\sigma}}\tau)(d\rtimes_{\widehat{\sigma}}1^0 )) \\
& =\sum_{i, j, i_1 , j_1}{}_{C_1} \la u_i \otimes F(u_i^* \cdot 1 \cdot v_j )\otimes[d^* E_1^{\sigma, v}(c^* v_j )]^{\widetilde{}}
\, , \,u_{i_1}\otimes F(u_{i_1}^* \cdot 1\cdot v_{j_1})\otimes \widetilde{v_{j_1}} \ra \\
& =\sum_{i, j, i_1 , j_1}{}_{C_1} \la u_i \cdot {}_A \la F(u_i^* \cdot 1\cdot v_j )\otimes
[d^* E_1^{\sigma, v}(c^* v_j )]^{\widetilde{}}
\, , \, F(u_{i_1}^* \cdot 1\cdot v_{j_1})\otimes\widetilde{v_{j_1}} \ra \, , \, u_{i_1} \ra \\
& =\sum_{i, j, i_1, j_1}{}_{C_1} \la u_i \cdot {}_A \la F(u_i^* \cdot 1\cdot v_j )\cdot \la d^* E_1^{\sigma, v}(c^* v_j )
\, , \, v_{j_1} \ra_A \, , \, F(u_{i_1}^* \cdot 1\cdot v_{j_1}) \ra \, , \, u_{i_1} \ra \\
& =\sum_{i, j, i_1 , j_1 }{}_{C_1} \la u_i \cdot {}_A \la F(u_i^* \cdot 1\cdot v_j )\cdot E_1^{\sigma, v}(E_1^{\sigma, v}(v_j^*c)
dv_{j_1}) \, , \, F(u_{i_1}^* \cdot 1\cdot v_{j_1}) \ra \, , \, u_{i_1} \ra \\
& =\sum_{i, j, i_1 , j_1}{}_{C_1} \la u_i \cdot {}_A \la F(u_i^* \cdot 1\cdot v_j )\cdot E_1^{\sigma, v}(v_j^* c )E_1^{\sigma v}
(dv_{j_1}) \, , \, F(u_{i_1}^* \cdot 1 \cdot v_{j_1}) \ra \, , \, u_{i_1} \ra \\
& =\sum_{i, j, i_1 , j_1}{}_{C_1} \la u_i \cdot {}_A \la F(u_i^* \cdot 1\cdot v_j E_1^{\sigma, v}(v_j^* c)) \, , \, F(u_{i_1}^*
\cdot 1\cdot v_{j_1}E_1^{\sigma, v}(v_{j_1}^* d^* )) \ra \, , \, u_{i_1} \ra \\
& =\sum_{i, i_1}{}_{C_1} \la u_i \cdot {}_A \la F(u_i^* \cdot 1\cdot c) \, , \, F(u_{i_1}^* \cdot 1\cdot d^* ) \ra \, , \, u_{i_1} \ra \\
& =\sum_{i, i_1}{}_{C_1} \la u_i \cdot F(u_i^* \cdot 1\cdot c)F(u_{i_1}^* \cdot 1\cdot d^* )^* \, , \, u_{i_1} \ra \\
& =\sum_{i, i_1}[u_i \cdot F(u_i^* \cdot 1\cdot c)F(u_{i_1}^* \cdot 1 \cdot d^* )^* ]
(1\rtimes_{\rho, u}1\rtimes_{\widehat{\rho}}\tau)u_{i_1}^* .
\end{align*}
Hence
\begin{align*}
& (E_2^{\rho, u}\circ\Psi_1 )((c\rtimes_{\widehat{\sigma}}1^0 )
(1\rtimes_{\sigma, v}1\rtimes_{\widehat{\sigma}}\tau)(d\rtimes_{\widehat{\sigma}}1^0 )) \\
& =\sum_{i, i_1}\frac{1}{N}u_i F(u_i^* \cdot 1\cdot c)F(u_{i_1}^* \cdot 1\cdot d^* )^* u_{1_1}^* \\
& =\frac{1}{N}\sum_{i_1}(1\cdot c)(u_{i_1}F(u_{i_1}^* \cdot 1\cdot d^* ))^*
=\frac{1}{N}(1\cdot c)(1\cdot d^* )^* .
\end{align*}
On the other hand,
\begin{align*}
& (\Psi\circ E_2^{\sigma, v})((c\rtimes_{\widehat{\sigma}}1^0 )(1\rtimes_{\sigma, v}1\rtimes_{\widehat{\sigma}}\tau)
(d\rtimes_{\widehat{\sigma}}1^0 )) \\
& =\frac{1}{N}\Psi((c\rtimes_{\widehat{\sigma}}1^0 )(d\rtimes_{\widehat{\sigma}}1^0 ))
= {}_C \la 1_A \cdot \frac{1}{N}cd \, , \, 1_A \ra \\
& =\frac{1}{N}{}_C \la 1_A \cdot c \, , \, 1_A \cdot d^* \ra =\frac{1}{N}(1_A \cdot c )(1_A \cdot d^* )^* .
\end{align*}
Therefore, $E_1^{\rho, u}\circ\Psi_1 =\Psi\circ E_1^{\sigma, v}$ since $D_1$ is the linear span of elements
$(c\rtimes_{\widehat{\sigma}}1^0 )(1\rtimes_{\widehat{\sigma}}\tau)(d\rtimes_{\widehat{\sigma}}1^0 )$, where
$c, d\in D$.

\begin{lemma}\label{lem:equal}With the above notations, $\widehat{\Psi}=\Psi_1$ on $D_1$.
\end{lemma}
\begin{proof}Since $D_1$ is the linear span of elements
$(c\rtimes_{\widehat{\sigma}}1^0 )(1\rtimes_{\widehat{\sigma}}\tau)(d\rtimes_{\widehat{\sigma}}1^0 )$,
where $c, d\in D$ and $\widehat{\Psi}=\Psi_1$ on $D$, we have to show that
$$
\Psi_1 (1\rtimes_{\widehat{\sigma}}\tau)=\widehat{\Psi}(1\rtimes_{\widehat{\sigma}}\tau)=1\rtimes_{\widehat{\rho}}\tau .
$$
Indeed, since $1_A$ is regarded as an element 
$$
\sum_{i, j}u_i \otimes F(u_i^* \cdot 1_A \cdot v_j )\otimes\widetilde{v_j}
$$
in $Y_1 (=C\otimes_A A \otimes_A \widetilde{D})$, where $\{(u_i , u_i^* )\}$ and $\{(v_j , v_j^* )\}$ are as above.
Thus
$$
\Psi_1 (1\rtimes_{\widehat{\sigma}}\tau)={}_C \la 1_A \cdot (1\rtimes_{\widehat{\sigma}}\tau) \, , \, 1_A \ra
={}_C \la 1_A \cdot (1\rtimes_{\widehat{\sigma}}\tau) \, , \, 1_A \cdot (1\rtimes_{\widehat{\sigma}}\tau) \ra .
$$
Here
\begin{align*}
1_A \cdot (1\rtimes_{\widehat{\sigma}}\tau) & =\sum_{i, j}u_i \otimes F(u_i^* \cdot 1_A \cdot v_j )\otimes
\widetilde{v_j}\cdot (1\rtimes_{\widehat{\sigma}}\tau) \\
& =\sum_{i, j}u_i \otimes F(u_i^* \cdot 1_A \cdot v_j )\otimes[(1\rtimes_{\widetilde{\sigma}}\tau)\cdot v_j ]^{\widetilde{}} \\
& =\sum_{i, j}u_i \otimes F(u_i^* \cdot 1_A \cdot v_j E_1^{\sigma, v}(v_j^* ))\otimes \widetilde{1_A} \\
& =\sum_i u_i \otimes F(u_i^* \cdot 1_A )\otimes\widetilde{1_A} \\
& =\sum_i u_i E_1^{\rho, u}(u_i^* )\otimes 1_A \otimes \widetilde{1_A} \\
& =1_A \otimes 1_A \otimes \widetilde{1_A} .
\end{align*}
Hence
$$
\Psi_1 (1\rtimes_{\widetilde{\sigma}}\tau)={}_C \la 1_A \otimes 1_A \otimes \widetilde{1_A} \, , \,
1_A \otimes 1_A \otimes \widetilde{1_A} \ra ={}_C \la 1_A \, , \, 1_A \ra =1\rtimes_{\widehat{\rho}}\tau .
$$
Therefore, we obtain the conclusion.
\end{proof}
By the above lemma, we obtain the following corollary:

\begin{cor}\label{cor:equal2}With the above notations, there is an isomorphism $\widehat{\Psi}$
of $D_1$ onto $C_1$ satisfying that
\begin{align*}
& \widehat{\Psi}|_D =\Psi , \quad \Psi\circ E_2^{\sigma, v}=E_2^{\rho, u}\circ\widehat{\Psi} , \\
& \widehat{\Psi}(1\rtimes_{\widehat{\sigma}}\tau)=1\rtimes_{\widehat{\rho}}\tau , \\
& \widehat{\Psi}(d_1 )={}_{C_1} \la 1_A \cdot d_1 \, , \, 1_A \ra \quad \text{for any $d_1 \in D_1$} .
\end{align*}
\end{cor}
 
In the same way as Corollary \ref{cor:equal2}, we obtain the following lemma:

\begin{lemma}\label{lem:equal3}With the above notations, let $\widehat{\widehat{\Psi}}$ be the isomorphism of
$D_2 (=D_1 \rtimes_{\widehat{\widehat{\sigma}}}H)$ onto $C_2 (=C_1\rtimes_{\widehat{\widehat{\rho}}}H)$
defined by
$$
\widehat{\widehat{\Psi}}(T)=\widehat{\Psi}\circ T\circ\widehat{\Psi^{-1}}
$$
for any $T\in\BB_D (D_1 )$, where we identify $C_2$ and $D_2$ with $\BB_C (C_1 )$, the $C^*$-algebra
of all adjointable right $C$-module maps on $C_1$ and $\BB_D (D_1 )$, the $C^*$-algebra of all adjointable
right $D$-module maps on $D_1$, respectively and we regard $C_1$ and $D_1$ as a $C_2 -C$-
equivalence bimodule and a $D_2 -D$-equivalence bimodule using the canonical conditional expectations
$E_2^{\rho, u} : C_1\to C$ and $E_2^{\sigma, v} : D_1 \to D$, respectively. Then $\widehat{\widehat{\Psi}}$
satisfying that
\begin{align*}
& \widehat{\widehat{\Psi}}|_{D_1}=\widehat{\Psi} , \quad \widehat{\Psi}\circ E_3^{\sigma, v}=E_3^{\rho, u}
\circ\widehat{\widehat{\Psi}} , \\
& \widehat{\widehat{\Psi}}(1\rtimes_{\widehat{\widehat{\sigma}}}e)=1\rtimes_{\widehat{\widehat{\rho}}}e , \\
& \widehat{\widehat{\Psi}}(d_2 )={}_{C_2} \la 1_A \cdot d_2 \, , \, 1_A \ra \quad\text{for any $d_2 \in D_2$},
\end{align*}
where $E_3^{\rho, u}$ and $E_3^{\sigma, v}$ are the canonical conditional expectations from $C_2$
and $D_2$ onto $C_1$ and $D_1$ defined by
$$
E_3^{\rho, u}(c_1 \rtimes_{\widehat{\widehat{\rho}}}h)=c_1 \tau(h) \,, \quad
E_3^{\sigma, v}(d_1\rtimes_{\widehat{\widehat{\sigma}}}h)=d_1 \tau(h)
$$
for any $c_1 \in C_1$, $d_1\in D_1$, $h\in H$, respectively.
\end{lemma}

By Lemmas \ref{lem:s-iso}, \ref{lem:equal3}, $Y_1 \cong C_{1 \widehat{\Psi}}$ as $C_1 -D_1$-equivalence bimodules
and $Y_2 \cong C_{2 \widehat{\widehat{\Psi}}}$ as $C_2 -D_2$-equivalence bimodules. Also, by \cite [Lemma 5.10]
{KT5:Hopf}, there is a $C^*$-Hopf algebra automorphism $f^0$ of $H^0$ such that
$$
\widehat{\widehat{\rho}}\circ\widehat{\Psi}=(\widehat{\Psi}\otimes f^0 )\circ\widehat{\widehat{\sigma}} .
$$

\begin{lemma}\label{lem:SME}With the above notations, let $\widehat{\widehat{\sigma}}_{f^0}=(\id\otimes f^0 )
\circ\widehat{\widehat{\sigma}}$. Then $\widehat{\widehat{\rho}}$ and $\widehat{\widehat{\sigma}}_{f^0}$ are
strongly Morita equivalent.
\end{lemma}
\begin{proof}Let $\lambda_{\widehat{\widehat{\rho}}}$ be the linear map from $C_{1 \widehat{\Psi}}$ to
$C_{1 \widehat{\Psi}}\otimes H^0$ defined by
$$
\lambda_{\widehat{\widehat{\rho}}}(x)=\widehat{\widehat{\rho}}(x)
$$
for any $x\in C_{1 \widehat{\Psi}}$. Then $\lambda_{\widehat{\widehat{\rho}}}$ is a coaction of $H^0$ on
$C_{1 \widehat{\Psi}}$ with respect to $(C_1 , D_1 , \widehat{\widehat{\rho}}, \widehat{\widehat{\sigma}}_{f^0})$.
Indeed, for any $c\in C_1$, $d\in D_1$, $x\in C_{1 \widehat{\Psi}}$,
\begin{align*}
\lambda_{\widehat{\widehat{\rho}}}(c\cdot x) & =\lambda_{\widehat{\widehat{\rho}}}(cx)=\widehat{\widehat{\rho}}(c)
\widehat{\widehat{\rho}}(x)=\widehat{\widehat{\rho}}(c)\cdot\lambda_{\widehat{\widehat{\rho}}}(x) , \\
\lambda_{\widehat{\widehat{\rho}}}(x\cdot d) & =\lambda_{\widehat{\widehat{\rho}}}(x\widehat{\Psi}(d))
=\widehat{\widehat{\rho}}(x)\widehat{\widehat{\rho}}(\widehat{\Psi}(d))=\widehat{\widehat{\rho}}(x)
\widehat{\Psi}(\widehat{\widehat{\sigma}}_{f^0}(d))=\lambda_{\widehat{\widehat{\rho}}}(x)\cdot
\widehat{\widehat{\sigma}}_{f^0}(d) .
\end{align*}
Also, for any $x, y\in C_{1 \widehat{\Psi}}$,
\begin{align*}
{}_{C_1 \otimes H^0} \la \lambda_{\widehat{\widehat{\rho}}}(x) \, , \, \lambda_{\widehat{\widehat{\rho}}}(y) \ra & =
\widehat{\widehat{\rho}}(x)\widehat{\widehat{\rho}}(y)^* =\widehat{\widehat{\rho}}(xy^* )=\widehat{\widehat{\rho}}
({}_{C_1} \la x, y \ra) , \\
\la \lambda_{\widehat{\widehat{\rho}}}(x) \, , \, \lambda_{\widehat{\widehat{\rho}}}(y) \ra_{D_1 \otimes H^0} & =
(\widehat{\Psi}^{-1}\otimes\id_{H^0})(\widehat{\widehat{\rho}}(x^* y))
=\widehat{\widehat{\sigma}}_{f^0}(\widehat{\Psi}^{-1}(x^* y))
=\widehat{\widehat{\sigma}}_{f^0}(\la x, y \ra_{D_1}) .
\end{align*}
Furthermore,
\begin{align*}
(\lambda_{\widehat{\widehat{\rho}}}\otimes\id)\circ\lambda_{\widehat{\widehat{\rho}}} & =(\widehat{\widehat{\rho}}\otimes\id)
\circ\widehat{\widehat{\rho}}=(\id\otimes\Delta^0 )\circ\widehat{\widehat{\rho}}=(\id\otimes\Delta^0 )\circ
\lambda_{\widehat{\widehat{\rho}}} , \\
(\id\otimes\epsilon^0 )\circ\lambda_{\widehat{\widehat{\rho}}} & =(id\otimes\epsilon^0 )\circ\widehat{\widehat{\rho}}=\id .
\end{align*}
Hence we obtain the conclusion.
\end{proof}

\begin{lemma}\label{lem:s-iso2}With the above notations, $C_{2 \widehat{\widehat{\Psi}}}\cong C_{1 \widehat{\Psi}}
\rtimes_{\lambda_{\widehat{\widehat{\rho}}}}H$ as $C_2 - D_2$-equivalence bimodules.
\end{lemma}
\begin{proof}We note that $C_{2 \widehat{\widehat{\Psi}}}=(C_1\rtimes_{\widehat{\widehat{\rho}}}
H)_{\widehat{\widehat{\Psi}}}$. For any $x\in C_1$, $h\in H$, $(x\rtimes_{\widehat{\widehat{\rho}}}
h)_{\widehat{\widehat{\Psi}}}$ denotes the element in $C_{2 \widehat{\widehat{\Psi}}}$ induced by $x\in C_1$,
$h\in H$. Let $\Theta$ be the bijective map from $C_{2 \widehat{\widehat{\Psi}}}$ onto
$C_{1 \widehat{\Psi}}\rtimes_{\lambda_{\widehat{\widehat{\rho}}}}H$ defined by
$$
\Theta((x\rtimes_{\widehat{\widehat{\rho}}}h)_{\widehat{\widehat{\Psi}}})=x\rtimes_{\lambda_{\widehat{\widehat{\rho}}}}h
$$
for any $x\in C_1$, $h\in H$. By easy computations, $\Theta$ is a $C_2 -D_2$-equivalence bimodule isomorphism of
$C_{2 \widehat{\widehat{\Psi}}}$ onto $C_{1 \widehat{\Psi}}\rtimes_{\lambda_{\widehat{\widehat{\rho}}}}H$.
Therefore, we obtain the conclusion.
\end{proof}

Since $C_{2 \widehat{\widehat{\Psi}}}\cong Y_2$ as $C_2 -D_2$-equivalence bimodules and
$C_{1 \widehat{\Psi}}\cong Y_1$ as $C_1 -D_1$-equivalence bimodules, by the proofs of
Lemmas \ref{lem:SME}, \ref{lem:s-iso2}, there is a coaction $\lambda$ of $H^0$ on $Y_1$
with respect to $(C_1 , D_1 , \widehat{\widehat{\rho}},  \, \widehat{\widehat{\sigma}}_{f^0})$ such that
$$
Y_2 \cong Y_1\rtimes_{\lambda}H
$$
as $C_2 -D_2$-equivalence bimodules. Therefore, we obtain the following proposition:

\begin{prop}\label{prop:s-iso3}Let $H$ be a finite dimensional $C^*$-Hopf algebra
with its dual $C^*$-Hopf algebra $H^0$. Let $(\rho, u)$ and $(\sigma, v)$ be twisted coactions of $H^0$ on
a unital $C^*$-algebra $A$. Let $C=A\rtimes_{\rho, u}H$, $D=A\rtimes_{\sigma, v}H$. We suppose that $A' \cap C=\BC1$.
(Hence $A' \cap D=\BC 1$.) Also, we suppose that the unital inclusions $A\subset C$ and $A\subset D$
are strongly Morita equivalent with respect to a $C-D$-equivalence bimodule $Y$ and its closed subspace $A$.
Let $C_1$, $C_2$, $D_1$, $D_2$, $\widehat{\widehat{\rho}}$, $\widehat{\widehat{\sigma}}$,
$\widehat{\widehat{\sigma_{f^0}}}$
and $Y_1$, $Y_2$ be as above.
Then there are a $C^*$-Hopf algebra automorphism $f^0$ of $H^0$ and a coaction $\lambda$ of
$H^0$ on $Y_1$ with respect to $(C_1, D_1 , \widehat{\widehat{\rho}}, \, \widehat{\widehat{\sigma_{f^0}}})$ such that
$$
Y_2 \cong Y_1 \rtimes_{\lambda}H
$$
as $C_2 -D_2$-equivalence bimodules.
\end{prop}

\section{The generalized Picard groups for coactions and the Picard groups for inclusions
of unital $C^*$-algebras}\label{sec:surjection}
In this section, we shall investigate the relation between the generalized Picard groups for coactions
of a finite dimensional $C^*$-Hopf algebra on a unital $C^*$-algebra and the Picard groups for unital
inclusions of unital $C^*$-algebras induced by the coaction.
\par
Let $H$ be a finite dimensional $C^*$-Hopf algebra
with its dual $C^*$-Hopf algebra $H^0$. Let $(\rho, u)$ be a twisted
coaction of $H^0$ on a unital
$C^*$-algebra $A$. Then we have the unital inclusion of unital $C^*$-algebras, $A\subset A\rtimes_{\rho, u}H$.
Hence we can obtain the generalized Picard group $\GPic_H^{\rho, u}(A)$ for the twisted coaction $(\rho, u)$ and the
Picard group $\Pic(A, A\rtimes_{\rho, u}H)$ for the unital inclusion of unital $C^*$-algebras $A\subset A\rtimes_{\rho, u}H$.
\par
Let $(X, \lambda, f^0 )\in\GEqui_H^{\rho, u}(A)$. Then we obtain the element
$(X, X\rtimes_{\lambda}H)\in\Equi(A, A\rtimes_{\rho, u}H)$. Indeed, $X$ is an $A-A$-equivalence
and $X\rtimes_{\lambda}H$ is an $A\rtimes_{\rho, u}H-A\rtimes_{\rho, u}H$-equivalence bimodule.
Also, $A\rtimes_{\rho_{f^0}, u_{f^0}}H$ is isomorphic to $A\rtimes_{\rho, u}H$ by the isomorphism
$\pi_{f^0}$ of $A\rtimes_{\rho, u}H$ onto $A\rtimes_{\rho_{f^0}, u_{f^0}}H$ defined by
$$
\pi_{f^0}(a\rtimes_{\rho, u}h)=a\rtimes_{\rho_{f^0}, u_{f^0}}f(h)
$$
for any $a\in A$, $h\in H$, where $f$ is the $C^*$-Hopf algebra automorphism of $H$ induced by $f^0$. Furthermore,
by routine computations, we can see that $X$ and $X\rtimes_{\lambda}H$ satisfy Conditions (1), (2)
in \cite [Definition 2.1]{KT4:morita}. Thus we can define the map $\theta$ from $\GPic_H^{\rho, u}(A)$ to
$\Pic(A, A\rtimes_{\rho, u}H)$ by
$$
\theta([X, \lambda, f^0 ])=[X, X\rtimes_{\lambda}H ]
$$
for any $[X, \lambda, f^0 ]\in\GPic_H^{\rho, u}(A)$. We note that $\theta$ is well-defined by routine computations.
We show that $\theta$ is a homomorphism of $\GPic_H^{\rho, u}(A)$ to $\Pic(A, A\rtimes_{\rho, u}H)$.
Let $[X, \lambda, f^0 ]$, $[Y, \mu, g^0 ]\in\GPic_H^{\rho, u}(A)$. Then
\begin{align*}
\theta([X, \lambda, f^0 ][Y, \mu, g^0 ]) & =\theta([X\otimes_A Y \, , \, \lambda\otimes\mu_{f^0} \, , \, 
f^0 \circ g^0 ]) \\
& =[X\otimes_A Y \, , \, (X\otimes_A Y)\rtimes_{\lambda\otimes\mu_{f^0}}H] .
\end{align*}
On the other hand,
\begin{align*}
\theta([X, \lambda, f^0 ])\theta([Y, \mu, g^0 ]) & =[X, X\rtimes_{\lambda}H][Y, Y\rtimes_{\mu}H] \\
& =[X\otimes_A Y \, , \, (X\rtimes_{\lambda}H)\otimes_{A\rtimes_{\rho, u}H}(Y\rtimes_{\mu}H)] \\
& =[(X\rtimes_{\lambda}1)\otimes_A (Y\rtimes_{\mu}1) \,, \, (X\rtimes_{\lambda}H)\otimes_{A\rtimes_{\rho, u}H}
(Y\rtimes_{\mu}H)] .
\end{align*}
We note that $A\rtimes_{\rho_{f^0}, u_{f^0}}H$, $A\rtimes_{\rho_{g^0}, u_{g^0}}H$ and
$A\rtimes_{\rho_{f^0 \circ g^0}, u_{f^0 \circ g^0}}H$ are identified with $A\rtimes_{\rho, u}H$ by the isomorphisms
$\pi_{f^0}$, $\pi_{g^0}$ and $\pi_{f^0 \circ g^0}$ defined respectively as follows:
\begin{align*}
\pi_{f^0}(a\rtimes_{\rho, u}h) & =a\rtimes_{\rho_{f^0}, u_{f^0}}f(h) \\
\pi_{g^0}(a\rtimes_{\rho, u}h) & =a\rtimes_{\rho_{g^0}, u^{g^0}}g(h) \\
\pi_{f^0 \circ g^0 }(a\rtimes_{\rho, u}h) & =a\rtimes_{\rho_{f^0 \circ g^0}, u_{f^0 \circ g^0 }}(f\circ g)(h)
\end{align*}
for any $a\in A$, $h\in H$. Let $\Phi$ be the linear map from $(X\otimes_A Y)\rtimes_{\lambda\otimes\mu_{f^0}}H$
to $(X\rtimes_{\lambda}H)\otimes_{A\rtimes_{\rho, u}H}(Y\rtimes_{\mu}H)$ defined by
$$
\Phi((x\otimes y)\rtimes_{\lambda\otimes\mu_{f^0}}h)=(x\rtimes_{\lambda}1)\otimes(y\rtimes_{\mu}f^{-1}(h))
$$
for any $x\in X$, $y\in Y$, $h\in H$. For any $x\in X$, $y\in Y$, $h, l\in H$,
\begin{align*}
(x\rtimes_{\lambda}h)\otimes(y\rtimes_{\mu}l) & =(x\rtimes_{\lambda}1)\cdot (1\rtimes_{\rho_{f^0}, u_{f^0}}h)
\otimes(y\rtimes_{\mu}l) \\
& =(x\rtimes_{\lambda}1)\otimes(1\rtimes_{\rho, u}f^{-1}(h))\cdot(y\rtimes_{\mu}l) \\
& =(x\rtimes_{\lambda}1)\otimes([f^{-1}(h_{(1)})\cdot_{\mu}y]\widehat{u_{g^0}}(f^{-1}(h_{(2)}), l_{(1)})\rtimes_{\mu}f^{-1}
(h_{(3)})l_{(2)}) .
\end{align*}
Hence $(X\rtimes_{\lambda}H)\otimes_{A\rtimes_{\rho, u}H}(Y\rtimes_{\mu}H)$ is the closure of linear spans
of elements $(x\rtimes_{\lambda}1)\otimes(y\rtimes_{\mu}h)$ in $(X\rtimes_{\lambda}H)\otimes_{A\rtimes_{\rho, u}H}
(Y\rtimes_{\mu}H)$, where $x\in X$, $y\in Y$, $h\in H$. Thus $\Phi$ is surjective. Also, its inverse map
$\Phi^{-1}$ is following:
$$
\Phi^{-1}((x\rtimes_{\lambda}1)\otimes(y\rtimes_{\mu}h))=(x\rtimes y)\rtimes_{\rho\otimes\mu_{f^0}}f(h)
$$
for any $x\in X$, $y\in Y$, $h\in H$. Furthermore, let $x, z\in X$, $y, w\in Y$, $h, l\in H$. Then
\begin{align*}
& {}_{A\rtimes_{\rho, u}H} \la (x\otimes y)\rtimes_{\lambda\otimes\mu_{f^0}}h \, , \,
(z\otimes w)\rtimes_{\lambda\otimes\mu_{f^0}}l \ra \\
& ={}_A \la x\otimes y \, , \, [S(h_{(2)}l_{(3)}^* )^* \cdot_{\lambda\otimes\mu_{f^0}}(z\otimes w)]
\cdot\widehat{u_{f^0 \circ g^0}}(S(h_{(1)}l_{(2)}^* )^* \, , \, l_{(1)}) \rtimes_{\rho, u}h_{(3)}l_{(4)} \\
& ={}_A \la x\otimes y \, , \, [S(h_{(3)}l_{(4)}^* )^* \cdot_{\lambda}z]\otimes[S(h_{(2)}l_{(3)}^* )^* \cdot_{\mu_{f^0}}w]
\cdot\widehat{u_{f^0 \circ g^0}}(S(h_{(1)}l_{(2)}^* )^* \, , \, l_{(1)}) \\
& \rtimes_{\rho, u}h_{(4)}l_{(5)}^* \\
& ={}_A \la x\cdot {}_A \la y \, , \, [S(h_{(2)}l_{(3)}^* )^* \cdot_{\mu_{f^0}}w]\cdot
\widehat{u_{f^0 \circ g^0}}(S(h_{(1)}l_{(2)}^* )^* , l_{(1)}) \ra \, , \, [S(h_{(3)}l_{(4)}^* )^* \cdot_{\lambda}z] \ra \\
& \rtimes_{\rho, u}h_{(4)}l_{(5)}^* .
\end{align*}
On the other hand,
\begin{align*}
& {}_{A\rtimes_{\rho, u}H} \la \Phi((x\otimes y)\rtimes_{\lambda\otimes\mu_{f^0}}h) \, , \, \
\Phi((z\otimes w)\rtimes_{\lambda\otimes\mu_{f^0}}l )\ra \\
& ={}_{A\rtimes_{\rho, u}H} \la (x\rtimes_{\lambda}1)\otimes(y\rtimes_{\mu}f^{-1}(h)) \, , \,
(z\rtimes_{\lambda}1)\otimes(w\rtimes_{\mu}f^{-1}(l)) \ra \\
& ={}_{A\rtimes_{\rho, u}H} \la (x\rtimes_{\lambda}1)\cdot {}_{A\rtimes_{\rho, u}H} \la y\rtimes_{\mu}f^{-1}(h) \, , \, 
w\rtimes_{\mu}f^{-1}(l) \ra \, , \, z\rtimes_{\lambda}1 \ra \\
& ={}_{A\rtimes_{\rho, u}H}(x\rtimes_{\lambda}1)\cdot {}_A \la y \, , \, [f^{-1}(S(h_{(2)}l_{(3)}^* )^* )\cdot_{\mu}w]\cdot
\widehat{u_{g^0}}(f^{-1}(S(h_{(1)}l_{(2)}^* )^* ) \, , \, f^{-1}(l_{(1)})) \ra \\
& \rtimes_{\rho, u}f^{-1}(h_{(3)}l_{(4)}^* ) \, , \, z\rtimes_{\lambda}1 \ra \\
&  ={}_{A\rtimes_{\rho, u}H} \la (x\rtimes_{\lambda}1)\cdot {}_A \la y \, , \, [S(h_{(2)}l_{(3)}^* )^* \cdot_{\mu_{f^0}}w]\cdot
\widehat{u_{f^0 \circ g^0}}(S(h_{(1)}l_{(2)}^* )^* ) \, , \, l_{(1)} \ra \\
& \rtimes_{\rho, u}f^{-1}(h_{(3)}l_{(4)}^* ) \, , \, z\rtimes_{\lambda}1 \ra \\
& ={}_{A\rtimes_{\rho, u}H} \la x\, \cdot \, {}_A \la y \, , \, [S(h_{(2)}l_{(3)}^* )^* \cdot_{\mu_{f^0}}w]
\cdot{\widehat{u_{f^0 \circ g^0}}}
(Sh_{(1)}l_{(2)}^* )^* \, , \, l_{(1)}) \ra \\
& \rtimes_{\lambda}h_{(3)}l_{(4)}^* \, , \, z\rtimes_{\lambda}1 \ra \\
& ={}_A \la x\cdot {}_A \la y \, , \, [S(h_{(2)}l_{(3)}^* )^* \cdot_{\mu_{f^0}}w]\cdot\widehat{u_{f^0 \circ g^0}}
(S(h_{(1)}l_{(2)}^* )^* \, , \, l_{(1)}) \ra \, , \, [S(h_{(3)}l_{(4)}^* )^* \cdot_{\lambda}z] \ra \\
& \rtimes_{\rho, u}h_{(4)}l_{(5)}^* .
\end{align*}
Hence $\Phi$ preserves the left $A\rtimes_{\rho, u}H$-inner products. Also,
\begin{align*}
& \la (x\otimes y)\rtimes_{\lambda\otimes\mu_{f^0}}h \, , \, (z\otimes w)\rtimes_{\lambda\otimes\mu_{f^0}}l
\ra_{A\rtimes_{\rho, u}H} \\
& =\widehat{u_{f^0 \circ g^0}^*}(h_{(2)}^* \, , \, S(h_{(1)})^* )[h_{(3)}^* \cdot_{\rho_{f^0 \circ g^0}, u_{f^0 \circ g^0}}
\la x\otimes y \, ,\, z\otimes w \ra_A ]\widehat{u_{f^0 \circ g^0}}(h_{(4)}^* \, , \, l_{(1)}) \\
& \rtimes_{\rho, u}(g^{-1}\circ f^{-1})
(h_{(5)}^* l_{(2)}) \\
& =\widehat{u_{f^0 \circ g^0}^*}(h_{(2)}^* \, , \, S(h_{(1)}^* ))[h_{(3)}^* \cdot_{\rho_{f^0 \circ g^0}, u_{f^0 \circ g^0}}
\la y \, , \, \la x, z \ra_A \cdot w \ra_A ]\widehat{u_{f^0 \circ g^0}}(h_{(4)}^* \, , \, l_{(1)}) \\
& \rtimes_{\rho, u}(g^{-1}\circ f^{-1})(h_{(5)}^* l_{(2)}) .
\end{align*}
On the other hand, 
\begin{align*}
& \la \Phi((x\otimes y)\rtimes_{\lambda\otimes\mu_{f^0}}h) \, , \, \Phi((z\otimes w)\rtimes_{\lambda\otimes\mu_{f^0}}l)
\ra_{A\rtimes_{\rho, u}H} \\
& \la (x\rtimes_{\lambda} 1)\otimes (y\rtimes_{\mu}f^{-1}(h)) \, , \, (z\rtimes_{\lambda}1)\otimes(w\rtimes_{\mu}f^{-1}(l))
\ra_{A\rtimes_{\rho, u}H} \\
& =\la y\rtimes_{\mu}f^{-1}(h) \, , \, \la x\rtimes_{\lambda}1 \, , \, z\rtimes_{\lambda}1 \ra_{A\rtimes_{\rho, u}H}
\cdot w\rtimes_{\mu}f^{-1}(l) \ra_{A\rtimes_{\rho, u}H} \\
& =\la y\rtimes_{\mu}f^{-1}(h) \, , \, \la x, z \ra_A \rtimes_{\rho, u}1\cdot w\rtimes_{\mu}f^{-1}(l) \ra_{A\rtimes_{\rho, u}H} \\
& =\la y\rtimes_{\mu}f^{-1}(h) \, , \, \la x, y \ra_A \cdot w \rtimes_{\mu}f^{-1}(l) \ra_{A\rtimes_{\rho, u}H} \\
& =\widehat{u_{g^0}^*}(f^{-1}(h_{(2)}^* ) \, , \, f^{-1}(S(h_{(1)}^* )))[f^{-1}(h_{(3)}^* )\cdot_{\rho_{g^0}, \, u_{g^0}}
\la y\, , \, \la x, z \ra_A \cdot w \ra_A ] \\
& \times\widehat{u_{g^0}}(f^{-1}(h_{(4)}^* ) \, , \, f^{-1}(l_{(1)}))\rtimes_{\rho, u}g^{-1}(f^{-1}(h_{(5)}^* l_{(2)})) \\
& =\widehat{u_{f^0 \circ g^0}^*}(h_{(2)}^* \, , \, S(h_{(1)}^* ))[h_{(3)}^* \cdot_{\rho_{f^0 \circ g^0}, u_{f^0 \circ g^0}}
\la y \, , \, \la x, z \ra_A \cdot w \ra_A ]\widehat{u_{f^0 \circ g^0}}(h_{(4)}^* \, , \, l_{(1)}) \\
& \rtimes_{\rho, u}(g^{-1}\circ f^{-1})(h_{(5)}^* l_{(2)}) .
\end{align*}
Hence $\Phi$ preserves the right $A\rtimes_{\rho, u}H$-inner products. Therefore, by the remark after Jensen Thomsen
\cite [Definition 1.1.18]{JT:KK}, $\Phi$ is  an $A\rtimes_{\rho, u}H-A\rtimes_{\rho, u}H$-equivalence bimodule
isomorphism of $(X\otimes_A Y)\rtimes_{\lambda\otimes\mu_{f^0}}H$ onto
$(X\rtimes_{\lambda}H)\otimes_{A\rtimes_{\rho, u}H}(Y\rtimes_{\mu}H)$. Furthermore,
$$
\Phi((X\otimes_A Y)\rtimes_{\lambda\otimes\mu_{f^0}}1)=
(X\rtimes_{\lambda}1)\otimes_{A\rtimes_{\rho, u}H}(Y\rtimes_{\mu}1)
$$
by the definition of $\Phi$. Therefore $\theta$ is a homomorphism of $\GPic_H^{\rho, u}(A)$ to
$\Pic(A, A\rtimes_{\rho, u}H)$.
\par
We shall show that $\theta$ is surjective if $A' \cap (A\rtimes_{\rho, u}H)=\BC1$.
Let $C=A\rtimes_{\rho, u}H$ and we suppose that $A' \cap C=\BC 1$. Let $[X, Y]\in\Pic(A, C)$. Then
$[Y, Y_1 ]\in\Pic(C, C_1 )$, where $C_1 =C\rtimes_{\widehat{\rho}}H^0$ and $Y_1$ is the upward
basic construction of $Y$ for $E^X$, the unique conditional
expectation from $Y$ onto $X$ with respect to $E_1^{\rho, u}$ and $E_1^{\rho, u}$ (See \cite [Definition 6.1]{KT4:morita}).
Also, by \cite [Section 4]{KT5:Hopf},
there are a coaction $\beta$ of $H$ on $C$ and a coaction $\mu$ of $H$ on $Y$ such that $\widehat{\rho}$ and $\beta$
are strongly Morita equivalent with respect to the coaction $\mu$, that is, $\mu$ is a coaction of $H$ on $Y$
with respect to $(C, C, \widehat{\rho}, \beta)$. Hence the unital inclusions $C\subset C_1$ and $C\subset D_1$ are
strongly Morita equivalent to with respect to the $C_1 -D_1$-equivalence bimodule $Y\rtimes_{\mu}H^0$
and its closed subspace $Y$, where $D_1 =C\rtimes_{\beta}H^0$. Since $[Y, Y_1 ]\in \Pic(C, C_1 )$,
the unital inclusions $C\subset C_1$ and $C \subset D_1$ are strongly Morita equivalent with respect to the
$C_1 -D_1$-equivalence bimodule $\widetilde{Y_1}\otimes_{C_1}(Y\rtimes_{\mu}H^0 )$ and its closed
subspace $\widetilde{Y}\otimes_C Y \cong {}_C C_C$. We note that $C' \cap C_1 =\BC 1$ since $A' \cap C=\BC 1$.
Let $Z=\widetilde{Y_1}\otimes_{C_1}(Y\rtimes_{\mu}H^0 )$.
Then there is the unique conditional expectation $F$ from $Z$ onto ${}_C C_C$ with respect to $E_2^{\rho, u}$ and
$E_2^{\rho, u}$ by \cite [Lemma 3.1]{Kodaka:Picard} since $A' \cap C=\BC 1$. By Proposition
\ref {prop:s-iso3}, there are a $C^*$-Hopf algebra automorphism $f$ of $H$ and a coaction $\lambda$ of $H$ on $Z_1$ with
respect to $(C_2 , D_2 , \widehat{\widehat{\widehat{\rho}}}, \widehat{\widehat{\beta_f}})$ such that 
$$
Z_2 \cong Z_1 \rtimes_{\lambda}H^0
$$
as $C_3 -D_3$-equivalence bimodules, where $Z_1$ is the upward basic construction of $Z$ for $F$
and $C_2 =C_1 \rtimes_{\widehat{\widehat{\rho}}}H$, $D_2 =D_1 \rtimes_{\widehat{\beta}}H$,
$C_3 =C_2 \rtimes_{\widehat{\widehat{\widehat{\rho}}}}H^0$, $D_3 =D_2 \rtimes_{\widehat{\widehat{\beta}}}H^0$,
$\beta_f =(\id\otimes f)\circ\beta$.
Let $Y_2$ and $Y_3$ be the upward and the second upward basic constructions of $Y_1$ for $E^Y$,
respectively. Then $[Y_2 , Y_3 ]\in\Pic(C_2 , C_3 )$. Let $\widehat{\widehat{\widehat{\theta}}}$ be the homomorphism
of $\GPic_{H^0}^{\widehat{\widehat{\widehat{\rho}}}}(C_2 )$ to $\Pic(C_2 , C_3 )$ in the same way as in the
beginning of this section. We show that 
there is an element $x\in\GPic_{H^0}^{\widehat{\widehat{\widehat{\rho}}}}(C_2 )$ such that
$$
\widehat{\widehat{\widehat{\theta}}}(x)=[Y_2 , Y_3 ]
$$
in $\Pic(C_2 , C_3 )$. In order to this, we prepare some lemmas: Let $(\rho, u)$ and $(\sigma, v)$ be
twisted coactions of $H^0$ on unital $C^*$-algebras $A$ and $B$, respectively.
We suppose that $(\rho, u)$ and $(\sigma, v)$ are strongly Morita equivalent. Let $\lambda$ be the twisted
coaction of $H^0$ on an $A-B$-equivalence bimodule $X$ with respect to $(A, B, \rho, u, \sigma, v)$.
Let $C=A\rtimes_{\rho, u}H$, $C_1 =C\rtimes_{\widehat{\rho}}H^0$,
$D=B\rtimes_{\sigma, v}H$, $D_1 =D\rtimes_{\widehat{\sigma}}H^0$, respectively.
Let $E^X$ be the conditional expectation from $X\rtimes_{\lambda}H$ onto $X$ for
$E_1^{\rho, u}$ and $E_1^{\sigma, v}$, which is defined by
$$
E^X (x\rtimes_{\lambda}h)=x\tau(h)
$$
for any $x\in X$, $h\in H$.

\begin{lemma}\label{lem:prepare1}With the above notations and assumptions, let $Y_1$ be the upward basic construction of
$X\rtimes_{\lambda}H$ for $E^X$. Then $Y_1 \cong X\rtimes_{\lambda}H\rtimes_{\widehat{\lambda}}H^0$ as
$C_1 -D_1$-equivalence bimodules.
\end{lemma}
\begin{proof}Let $E^{X\rtimes_{\lambda}H}$ be the conditional expectation from
$X\rtimes_{\lambda}H\rtimes_{\widehat{\lambda}}H^0$ onto $X\rtimes_{\lambda}H$ with respect to $E_2^{\rho, u}$ and
$E_2^{\sigma, v}$, which is defined by
$$
E^{X\rtimes_{\lambda}H}(x\rtimes_{\lambda}h\rtimes_{\widehat{\lambda}}\phi)=(x\rtimes_{\lambda}h)\phi(e)
$$
for any $x\in X$, $h\in H$. Hence by \cite [Proposition 6.11]{KT4:morita}, we obtain the conclusion.
\end{proof}

\begin{lemma}\label{lem:prepare5}With the above notations and assumptions, $(X\rtimes_{\lambda}H)^{\widetilde{}}$
is isomorphic to $\widetilde{X}\rtimes_{\widetilde{\lambda}}H$ as $B\rtimes_{\sigma , v}H-A\rtimes_{\rho, u}H$-
equivalence bimodules
\end{lemma}
\begin{proof}We can prove the lemma modifying the proof of \cite [Lemma 4.3]{Kodaka:equivariance}.
Let $\pi$ be the bijective linear map from $(X\rtimes_{\lambda}H)^{\widetilde{}}$ onto
$\widetilde{X}\rtimes_{\widetilde{\lambda}}H$ defined by
$$
\pi((x\rtimes_{\lambda}h)^{\widetilde{}})=\{[S(h_{(3)})\cdot_{\lambda}x]\widehat{v}(S(h_{(2)}), h_{(1)})\}^{\widetilde{}}
\rtimes_{\lambda}h_{(4)}^* 
$$
for any $x\in X$, $h\in H$. Then by routine computations, $\pi$ is a $B\rtimes_{\sigma, v}H-A\rtimes_{\rho, u}H$
-equivalence bimodule isomorphism of $(X\rtimes_{\lambda}H)^{\widetilde{}}$ onto
$\widetilde{X}\rtimes_{\widetilde{\lambda}}H$.
\end{proof}

We recall that $Z=\widetilde{Y_1}\otimes_{C_1}(Y\rtimes_{\mu}H^0 )$. Hence by Lemmas \ref{lem:prepare1} and
\cite [Lemma 3.3]{Kodaka:Picard},
$$
Z_1 \cong\widetilde{Y_2}\otimes_{C_2}(Y\rtimes_{\mu}H^0 \rtimes_{\widehat{\mu}}H)
$$
as $C_2 -D_2$-equivalence bimodules and
$$
Z_2 \cong \widetilde{Y_3}\otimes_{C_3}(Y\rtimes_{\mu}H^0 \rtimes_{\widehat{\mu}}H
\rtimes_{\widehat{\widehat{\mu}}}H^0 )
$$
Furthermore, by \cite [Lemmas 3.4 and 3.5]{Kodaka:Picard}
the inclusions $Z_1 \subset Z_2$ and
$$
\widetilde{Y_2}\otimes_{C_2}(Y\rtimes_{\mu}H^0 \rtimes_{\widehat{\mu}}H)\subset \widetilde{Y_3}\otimes_{C_3}
(Y\rtimes_{\mu}H^0 \rtimes_{\widehat{\mu}}H\rtimes_{\widehat{\widehat{\mu}}}H^0 )
$$
are strongly Morita equivalent as unital inclusions of unital $C^*$-algebras. Also, we recall that
by Proposition \ref{prop:s-iso3},
there are the $C^*$-Hopf algebra automorphism $f$ of $H$ and the coaction $\lambda$ of $H$ on $Z_1$
with respect to $(C_2 ,  \,D_2 , \, \widehat{\widehat{\widehat{\rho}}}, \, \widehat{\widehat{\beta_f}})$ such that
$$
Z_2 \cong Z_1 \rtimes_{\lambda}H^0
$$
as $C_3 -D_3$-equivalence bimodules. Hence
$$
Z_1 \rtimes_{\lambda}H^0 \cong \widetilde{Y_3}\otimes_{C_3}(Y\rtimes_{\mu}H^0 \rtimes_{\widehat{\mu}}H
\rtimes_{\widehat{\widehat{\mu}}}H^0 ) .
$$
as $C_3 -D_3$-equivalence bimodules. Thus by Lemma \ref{lem:prepare5},
\begin{align*}
Y_3 & \cong (Y\rtimes_{\mu}H^0 \rtimes_{\widehat{\mu}}H\rtimes_{\widehat{\widehat{\mu}}}H^0 )
\otimes_{D_3}(Z_1 \rtimes_{\lambda}H^0 )^{\widetilde{}} \\
& \cong [(Y\rtimes_{\mu}H^0 \rtimes_{\widehat{\mu}}H)\otimes_{Z_2}\widetilde{Z_1}]\rtimes_{\widehat{\widehat{\mu}}
\otimes\widetilde{\lambda}}H^0
\end{align*}
as $C_3 -C_3$-equivalence bimodules.
Also,
$Y_2 \cong (Y\rtimes_{\mu}H^0 \rtimes_{\widehat{\mu}}H)\otimes_{D_2}\widetilde{Z_1}$ as
$C_2 -C_2$-equivalence bimodules.
It follows that
$$
[Y_2 , Y_3 ]=[(Y\rtimes_{\mu}H^0 \rtimes_{\widehat{\mu}}H)\otimes_{D_2}\widetilde{Z_1} \, , \,
[(Y\rtimes_{\mu}H^0 \rtimes_{\widehat{\mu}}H)\otimes_{D_2}\widetilde{Z_1}]
\rtimes_{\widehat{\widehat{\mu}}\otimes\widetilde{\lambda}}H^0 ]
$$
in $\Pic(C_2 , C_3 )$. Then
$$
[(Y\rtimes_{\mu}H^0 \rtimes_{\widehat{\mu}}H)\otimes_{D_2}\widetilde{Z_1} \, , \, \widehat{\widehat{\mu}}
\otimes\widetilde{\lambda} \, , \,
f^{-1}]\in\Pic_{H^0}^{\widehat{\widehat{\widehat{\rho}}}}(C_2 )
$$
and
$$
\widehat{\widehat{\widehat{\theta}}}([(Y\rtimes_{\mu}H^0 \rtimes_{\widehat{\mu}}H)\otimes_{D_2}\widetilde{Z_1} \, , \,
\widehat{\widehat{\mu}}\otimes\widetilde{\lambda} \, , \, f^{-1}])=[Y_2 , Y_3 ] .
$$
Furthermore, we have the following lemma:

\begin{lemma}\label{lem:prepare6}Let $(\rho, u)$ be a twisted coaction of $H^0$ on a unital $C^*$-algebra $A$.
Let $\theta$ be the homomorphism of $\GPic_H^{\rho, u}(A)$ to $\Pic(A, C)$ defined by
$$
\theta([X, \lambda, f^0 ])=[X, X\rtimes_{\lambda}H]
$$
for any $(X, \lambda, f^0 )\in\GEqui_H^{\rho, u}(A)$, where $C=A\rtimes_{\rho, u}H$. Also, let $\widehat{\theta}$ be the
homomorphism of $\GPic_{H^0}^{\widehat{\rho}}(C)$ to $\Pic(C, C_1 )$ defined by
$$
\widehat{\theta}([Y, \mu, g])=[Y, Y\rtimes_{\mu}H^0 ]
$$
for any $(Y, \mu, g)\in\GEqui_{H^0}^{\widehat{\rho}}(C)$, where $C_1 =C\rtimes_{\widehat{\rho}}H^0$.
Furthermore, let $F$ and $G$ be the isomorphisms of $\GPic_H^{\rho, u}(A)$
and $\Pic(A, C)$ onto $\GPic_{H^0}^{\widehat{\rho}}(C)$ and $\Pic(C, C_1 )$ defined by
\begin{align*}
F([X, \lambda, f^0 ]) & =[X\rtimes_{\lambda}H, \widehat{\lambda}, f] , \\
G([X, Y]) & =[Y, Y_1 ]
\end{align*}
for any $(X, \lambda, f^0 )\in\GEqui_{H^0}^{\widehat{\rho}}(C)$ and $(X, Y)\in\Equi(A, C)$, respectively,
where $f$ is the $C^*$-Hopf algebra automorphism of $H$ induced by $f^0$ and $Y_1$ is the upward basic
construction of $Y$ for the conditional expectation $E^X :Y\to X$ with respect to the canonical
conditional expectation $E_1^{\rho, u} : C\to A$ and $E_1^{\rho, u}$. Then $G\circ\theta =\widehat{\theta}\circ F$.
\end{lemma}
\begin{proof}Let $[X, \lambda, f^0 ]$ be any element in $\GPic_H^{\rho, u}(A)$. Then
\begin{align*}
(G\circ\theta)([X, \lambda, f^0 ]) & =[X\rtimes_{\lambda}H \, ,  \, C\otimes_A X \otimes_A \widetilde{C}] , \\
(\widehat{\theta}\circ F)([X, \lambda, f^0 ]) & =
[X\rtimes_{\lambda}H \, , \, X\rtimes_{\lambda}H\rtimes_{\widehat{\lambda}}H^0 ] .
\end{align*}
By Lemma \ref {lem:prepare1} and \cite [Proposition 6.11]{KT4:morita}, we can see that
$$
[X\rtimes_{\lambda}H \, , \, C\otimes_A X\otimes_A \widetilde{C}]=[X\rtimes_{\lambda}H \, , \, X\rtimes_{\lambda}H
\rtimes_{\widehat{\lambda}}H^0 ]
$$
in $\Pic(C, C_1 )$. Therefore, we obtain the conclusion.
\end{proof}
\begin{thm}\label{thm:surjection}Let $\theta$ be the homomorphism of $\GPic_H^{\rho, u}(A)$ to $\Pic(A, C)$ defined
in the above. We suppose that $A' \cap C=\BC1$. Then $\theta$ is surjective.
\end{thm}
\begin{proof}Let $[X, Y]$ be any element in $\Pic(A, C)$. Then by the discussions before Lemma \ref{lem:prepare6},
$$
\widehat{\widehat{\widehat{\theta}}}([Y\rtimes_{\mu}H^0 \rtimes_{\widehat{\mu}}H)\otimes\widetilde{Z_1} \, , \,
\widehat{\widehat{\mu}}\otimes\widetilde{\lambda} \, , \, f^{-1}]=[Y_2, Y_3] .
$$
Hence by Lemma \ref{lem:prepare6}, $\theta$ is surjective.
\end{proof}

Next, we shall show that
$$
\Ker \, \theta=\{[{}_A A_A \, , \, \rho, f^0 ]\in\GPic_H^{\rho, u}(A) \, | \, f^0 \in \Aut (H^0 ) \} .
$$
Clearly, for any $f^0 \in\Aut (H^0 )$, $[{}_A A_A \, , \, \rho \, , \, f^0 ]\in\Ker\, \theta$.
We show that
$$
\Ker \, \theta\subset\{[{}_A A_A \, , \, \rho, f^0 ]\in\GPic_H^{\rho, u}(A) \, | \, f^0 \in \Aut (H^0 ) \} .
$$
Let $[X, \lambda, f^0 ]\in\GPic_H^{\rho, u}(A)$. We suppose that $\theta([X, \lambda, f^0 ])=[A, C]$
in $\Pic_H^{\rho, u}(A)$. Then there is a $C-C$-equivalence bimodule isomorphism $\pi$ of $C$ onto
$X\rtimes_{\lambda}H$ such that $\pi|_A$ is an $A-A$-equivalence bimodule isomorphism
of $A$ onto $X$. Let $\pi_0 =\pi|_A$.

\begin{lemma}\label{lem:kernel1}With the above notations and assumptions, $[X, \lambda, f^0 ]
=[{}_A A_A \, , \, \rho, f^0 ]$ in $\GPic_H^{\rho, u}(A)$.
\end{lemma}
\begin{proof}For any $a\in A$, $h\in H$,
\begin{align*}
\pi_0 (h\cdot_{\rho, u}a) &=\pi((1\rtimes_{\rho, u}h_{(1)})(a\rtimes_{\rho, u}1)(1\rtimes_{\rho, u}S(h_{(2)})^* )^* ) \\
& =\pi((1\rtimes_{\rho, u}h_{(1)})\cdot (a\rtimes_{\rho, u}1)\cdot (1\rtimes_{\rho, u}S(h_{(2)})^* )^* ) \\
& =(1\rtimes_{\rho, u}h_{(1)})\cdot (\pi_0 (a)\rtimes_{\lambda}1)\cdot (1\rtimes_{\rho, u}S(h_{(2)})^* )^* \\
& =h\cdot_{\lambda}\pi_0 (a)
\end{align*}
by \cite [Lemma 3.12]{KT1:inclusion} and \cite [Lemma 5.6]{KT3:equivalence}. Hence
$[X, \lambda, f^0 ]=[{}_A A_A \, , \, \rho, f^0 ]$ in $\GPic_H^{\rho, u}(A)$.
\end{proof}

By Lemma \ref{lem:kernel1}, we have the following corollary:

\begin{cor}\label{cor:kernel2}With the above notations, we suppose that $A' \cap C=\BC1$. Then
$$
\Ker \, \theta=\{[{}_A A_A \, , \, \rho, f^0 ]\in\GPic_H^{\rho, u}(A) \, | \, f^0 \in \Aut (H^0 ) \} .
$$
\end{cor}

Furthermore, $\Ker \, \theta\cong \Aut (H^0 )$. Indeed, let $\kappa$ be the homomorphism of $\Aut(H^0 )$ to
$\GPic_H^{\rho, u}(A)$ defined in Remark \ref{rem:auto}(2). Then clearly $\Ima \kappa =\Ker \theta$. Hence
$\Ker \theta\cong \Aut(H^0 )$.

\begin{cor}\label{cor:kernel3}With the above notations, we suppose that $A' \cap C=\BC1$. Then
we have the exact sequence
$$
1\longrightarrow\Aut (H^0 )\overset{\kappa}\longrightarrow\GPic_H^{\rho, u}(A)\overset{\theta}\longrightarrow
\Pic(A, C)\longrightarrow 1 ,
$$
where $C=A\rtimes_{\rho, u}H$.
\end{cor}

\end{document}